\documentclass{amsart}

\usepackage{amsthm,amsmath,amssymb}
\usepackage{stmaryrd} 

\usepackage{hyperref}
\usepackage{cleveref}

\usepackage{todonotes}

\usepackage{tikz-cd}
\usepackage{pdflscape}
\usepackage{rotating}
\usepackage{afterpage}

\usepackage{booktabs}
\title{Coalgebraic Formal Curve Spectra and the Annular Tower}

\author{Eric Peterson}
\address{1 Oxford Street, Cambridge, MA 02143}
\email{ecp@math.harvard.edu}

\theoremstyle{plain}
\newtheorem*{theorem*}{Theorem}
\newtheorem*{lemma*}{Lemma}
\newtheorem*{corollary*}{Corollary}
\newtheorem{theorem}{Theorem}[subsection]
\newtheorem{lemma}[theorem]{Lemma}
\newtheorem{corollary}[theorem]{Corollary}
\theoremstyle{definition}
\newtheorem{definition}[theorem]{Definition}

\newtheorem{conjecture}[theorem]{Conjecture}

\newtheorem*{metaquestion*}{Meta-question}
\theoremstyle{remark}
\newtheorem{remark}[theorem]{Remark}

\numberwithin{equation}{section}
\newcommand{\Q}{\mathbb Q}
\newcommand{\R}{\mathbb R}
\newcommand{\C}{\mathbb C}
\newcommand{\Z}{\mathbb Z}
\newcommand{\F}{\mathbb F}
\newcommand{\W}{\mathbb W}

\newcommand{\X}{\mathbb X}
\renewcommand{\H}{\mathbb H}
\renewcommand{\S}{\mathbb S}
\newcommand{\RP}{\R\mathrm P}
\newcommand{\CP}{\C\mathrm P}
\newcommand{\HP}{\H\mathrm P}
\newcommand{\XP}{\X\mathrm P}
\newcommand{\m}{\mathfrak m}

\newcommand{\op}{\mathrm{op}}
\newcommand{\G}{\widehat{\mathbb G}}
\newcommand{\A}{\widehat{\mathbb A}}

\newcommand{\Susp}{\Sigma}
\newcommand{\Loops}{\Omega}
\newcommand{\sm}{\wedge}
\newcommand{\<}{\langle}
\renewcommand{\>}{\rangle}
\renewcommand{\epsilon}{\varepsilon}
\newcommand{\eps}{\epsilon}
\newcommand{\cotensor}{\mathbin{\square}}
\newcommand{\from}{\leftarrow}
\newcommand{\Cobar}{\Omega}
\newcommand{\Spin}{\mathrm{Spin}}
\newcommand{\String}{\mathrm{String}}

\newcommand{\co}{\colon\thinspace}

\DeclareMathOperator{\height}{ht}
\DeclareMathOperator{\Tor}{Tor}
\DeclareMathOperator{\Cotor}{Cotor}
\DeclareMathOperator{\End}{End}

\DeclareMathOperator{\Tot}{Tot}
\DeclareMathOperator{\Sch}{Sch}
\DeclareMathOperator{\cofib}{cofib}
\DeclareMathOperator{\fib}{fib}
\DeclareMathOperator{\Spf}{Spf}
\DeclareMathOperator{\colim}{colim}
\DeclareMathOperator{\Spec}{Spec}
\DeclareMathOperator{\Div}{Div}
\DeclareMathOperator{\Aut}{Aut}
\DeclareMathOperator{\SDiv}{SDiv}
\DeclareMathOperator{\Ext}{Ext}
\newcommand{\mylim}[2]{\mathop{\underset{#1}{\lim{}^{#2}}}}

\newcommand{\CatOf}[1]{\textsc{#1}}
\newcommand{\sheaf}[1]{\mathcal{#1}}
\newcommand{\range}[2]{{}_{#1}^{#2}}
\newcommand{\ps}[1]{\left\llbracket {#1} \right\rrbracket}

\renewcommand{\th}{{\textsuperscript{th}}}

\begin{document}

\begin{abstract}    
We import into homotopy theory the algebro-geometric construction of the cotangent space of a geometric point on a scheme.  Specializing to the category of spectra local to a Morava $K$--theory of height $d$, we show that this can be used to produce a choice-free model of the determinantal sphere as well as an efficient Picard-graded cellular decomposition of $K(\Z_p, d+1)$.  Coupling these ideas to work of Westerland, we give a ``Snaith's theorem'' for the Iwasawa extension of the $K(d)$--local sphere.
\end{abstract}

\maketitle


\section{Introduction}

Much of modern chromatic homotopy theory is underpinned by a process that converts spaces to formal schemes: given an even--periodic cohomology theory $E$ and a space $X$, if $E^* X$ is suitably nice we can define a formal scheme $X_E$ by the formula \[X_E := \Spf E^0 X,\] where the formal topology can be taken to come from the subskeleta of a cellular structure on $X$.  The very definition of complex--orientability is designed so that this construction carries $\CP^\infty$ to the formal affine line: \[\Spf E^0 \CP^\infty = \{\Spec E^0 \CP^n\}_n \cong \{\Spec E^0[x] / x^{n+1}\}_n =: \A^1_{/\Spec E^0}.\]  In fact, many familiar values of $X$ are carried to other familiar schemes---for instance, $BU(n)_E$ models the scheme of effective Weil divisors on $\CP^\infty_E$ of rank $n$, $\HP^\infty_E$ is sent to a formal curve, and the symplectification map is sent to a degree $2$ isogeny $\CP^\infty_E \to \HP^\infty_E$.  The reach of this construction is maximized when $E = E_\Gamma$ is taken to be the Morava $E$--theory associated to a formal group $\Gamma$ of finite height $d$ over a perfect field $k$ of positive characteristic $p$.  In this setting, a great many theorems have been proven about this assignment from spaces to formal schemes, as partially tabulated in \Cref{FormalSchemeListFig}.

\begin{figure}[h]
\begin{center}
\begin{tabular}{@{}cc@{}} \toprule
$X$ & $X_{E_\Gamma}$ \\
\cmidrule(l){1-1} \cmidrule(l){2-2}
one--point space & Lubin--Tate space for $\Gamma$ \\
$\CP^\infty$ & versal deformation $\tilde \Gamma$ of $\Gamma$ \\
$BU(n)$ & effective Weil divisors of rank $n$ on $\tilde \Gamma$ \\
$BU \times \Z$ & stable Weil divisors on $\tilde \Gamma$ \\
$BU$ & stable Weil divisors on $\tilde \Gamma$ of virtual rank $0$ \\
$BS^1[p^j]$ & the $p^j$--torsion subgroup $\tilde \Gamma[p^j]$ \\
$BA^*$ & the mapping scheme $\underline{\CatOf{FormalGroups}}(A, \tilde \Gamma)$ \\
$B^q(S^1)$ & the $q${\th} exterior power of $\tilde \Gamma$ \\
$B^q(S^1[p^j])$ & the $p^j$--torsion subgroup of $\tilde \Gamma^{\wedge q}$ \\
$BA^*$ (mod annihilators of $x(a)$) & level--$A$ structures on $\tilde \Gamma$ \\
$B\Sigma_n$ (mod transfers) & subgroup divisors of rank $n$ on $\tilde \Gamma$ \\
$BSU$ & special stable Weil divisors on $\tilde \Gamma$ of virtual rank $0$, \\
      & equivalently, $C_2(\tilde \Gamma) := \operatorname{Sym}^2_{\Div \tilde \Gamma} \Div_0 \tilde \Gamma$ \\
$BU[6, \infty)$ & $C_3(\tilde \Gamma)$ \\
$\HP^\infty$ & a formal curve $\overline \Gamma$ double-covered by $\tilde \Gamma$ \\
$B\mathrm{Sp}$ & stable Weil divisors on $\overline \Gamma$ \\
$BO$ & $\Div \overline \Gamma \times_{\Div \overline \Gamma[2]} \Div \tilde \Gamma[2]$ \\
$BSO$ & $\Div \overline \Gamma \times_{\Div \overline \Gamma[2]} \SDiv \tilde \Gamma[2]$ \\
$\Spin/SU$ & $C_2(\tilde \Gamma) / ([a, -a])$ \\
$B\Spin$ (when $\operatorname{ht} \Gamma \le 2$) & $C_2(\tilde \Gamma) / ([a, b] + [-a, -b])$ \\
$B\String$ (when $\operatorname{ht} \Gamma \le 2$) & $C_3(\tilde \Gamma) / ([a, b, -a-b])$ \\
$\vdots$ & $\vdots$ \\
\bottomrule
\end{tabular}
\end{center}
\caption{Some of the known association between spaces and formal schemes}\label{FormalSchemeListFig}
\end{figure}

The most remarkable feature of this table is that when $\Gamma$ varies and the space $X$ is fixed, the same family of formal schemes appears as output.  This inspires us to consider $X$ itself as playing some scheme--like role at the level of homotopy theory and to study algebro-geometric operations on $X$ intrinsically.  The goal of the present work is to use a spectrum-level construction analogous to that of Hopf algebra homology to define the ``tangent spectrum'' of certain extremely nice spaces $X$.  We first build up to this definition and then justify its utility by demonstrating that it has the expected behavior in $E_\Gamma$--homology.  The details of this computation inform us at what homotopical level we expect our construction to behave well: we find that it lives naturally in the category of spectra localized at $E_\Gamma$--cohomology, which we abbreviate to the category of \textit{$\Gamma$--local spectra}.  With this established, we then use the computational apparatus to extract further interesting constructions within the $\Gamma$--local category.  The main result of this paper is thus:
\vspace{\baselineskip}
\begin{theorem*}
For $X$ any pointed space and $\Gamma$ any finite height formal group over a perfect field $k$, there is a tower of $\Gamma$--local spectra, functorial in $X$:
\begin{center}
\begin{tikzcd}
C\range{0}{\infty} \arrow{r} & C\range1\infty \arrow{r} & C\range2\infty \arrow{r} & \cdots \\
C\range00 \arrow["\fib"]{u} & C\range11 \arrow["\fib"]{u} & C\range22 \arrow["\fib"]{u} & \cdots,
\end{tikzcd}
\end{center}
with natural equivalences
\begin{align*}
C\range00 & \simeq \S, &
C\range0\infty & \simeq \Susp^\infty_+ X, &
C\range1\infty & \simeq \Susp^\infty X.
\end{align*}
If $X_{E_\Gamma}$ is a formal curve (i.e., $E_\Gamma^0 X$ is abstractly isomorphic to a univariate formal power series $E_\Gamma^0$--algebra) and $p \gg \height \Gamma$,\footnote{This hypothesis comes from our proofs in \Cref{LimitAppendix} and seems likely to be unnecessary.} then \[T^*_0 (X_{E_\Gamma}) \cong E_\Gamma^0(C\range11),\] where $T_0^*$ indicates the cotangent space at the origin.  Moreover, $C\range{k}{k} = (C\range11)^{\sm k}$.
\end{theorem*}

In the case that $X_{E_\Gamma}$ is a formal curve, we are thus motivated to take $C\range11$ as a definition of $T_+ X$, the tangent space of $X$ at its natural pointing $\S^0 \to \Susp^\infty_+ X$.  This construction is most interesting when $X$ is chosen so as to have some special relevance to the $\Gamma$--local category.  For instance, we will show the following result, which is highly specific both to Morava $E$--theory and to the choice of $\Gamma$:

\begin{lemma*}
Setting $X = \Susp^\infty_+ B^d S^1$ for $d = \operatorname{ht} \Gamma$, $X_{E_\Gamma}$ is a formal curve and $\Aut \Gamma$ acts on the tangent space through the determinant representation.
\end{lemma*}

\begin{corollary*}
For $p \gg \height \Gamma$, the $\Gamma$--local spectrum $T_+ \Susp^\infty_+ B^d S^1$ gives a choice--free model of $\S^{\det}$, the determinantal sphere of Gross and Hopkins (cf.\ \cite[Remark 2.5]{GHMR}, \cite[Theorem 6 and Corollary 3]{HopkinsGrossAnnouncement}).
\end{corollary*}

\begin{corollary*}
For $p \gg \height \Gamma$, the spectrum $\Susp^\infty_+ B^d S^1$ has an efficient $\Gamma$--local Picard--graded cellular decomposition, with one cell of the form $C\range kk \simeq (\S^{\det})^{\sm k}$ for each $k \ge 0$.
\end{corollary*}

\noindent This is meant to be analogous to the motivic cellular decomposition of $\CP^\infty$ into even spheres, attached along homotopy classes graded by $\Susp^{-1} (\CP^1)^{\sm k}$.  Drawing from the algebro-geometric interpretation, we refer to this as the \textit{annular tower}.

Alternatively, the cellular decomposition of $\CP^\infty$ can also be interpreted as the skeletal filtration for the bar construction $B(\C^\times)$.  This, too, admits generalization:

\begin{corollary*}
There is an $A_\infty$--ring structure on $(\Susp^{-1} \S^{\det})_+$ and a $\Gamma$--local equivalence \[B((\Susp^{-1} \S^{\det})_+) \xrightarrow{\simeq} \Susp^\infty_+ B^d S^1.\]
\end{corollary*}

\subsection{Acknowledgements}

The paper covers the original results of my PhD thesis, completed in 2015 at the University of California, Berkeley, under the direction of Constantin Teleman (and with the support of the department's NSF grant for the geometry and topology group, DMS-0838703).  The germ of this project came from an unpublished note of Neil Strickland~\cite{StricklandConjecture}, and work on it (in particular, \Cref{RavenelWilsonForEthy}) began at the University of Illinois, Urbana-Champaign, under Matthew Ando (and with the support of UIUC's NSF grant for graduate students pursuing research projects, DMS-0838434).

This project was boosted considerably by two visits to the Max Planck Institut f\"ur Mathematik in Bonn, which was facilitated by Peter Teichner.  I also benefited significantly from many conversations with Michael Hopkins, with Hal Sadofsky, and especially with Tobias Barthel on the inverse limit spectral sequence presented here.  Rune Haugseng's anonymous referee for his paper on Morita categories~\cite{Haugseng} insisted that he include exactly the \(\infty\)--categorical results I needed to underpin this paper---so, I also extend a hearty thank you to this mysterious benefactor.  My own anonymous referee provided many helpful comments that greatly improved the quality of this paper, for which I am very grateful.  Finally, I would like to thank Samrita Dhindsa for her love and support throughout.

\subsection{Notational forewarnings}

With a fixed formal group $\Gamma$ in mind, we will often abbreviate the associated Morava $K$-- and $E$--theories from $K_\Gamma$ and $E_\Gamma$ to simply $K$ and $E$.  We will be especially interested in the behavior of limits taken in different categories, and so we will write $\lim{}_{\CatOf{C}}^{\alpha}(S_\alpha)$ to denote the limit in the category $\CatOf{C}$ of the $\alpha$--indexed pro-system $S_\alpha$.

\section{Homotopical coalgebras and comodules}
\label{SpectralCotensorSection}

\subsection{Motivation: coalgebraic formal schemes}

To set the stage, we expand on our discussion from the introduction of treating a space as a scheme--like object, starting with a more careful description of ``$X_E$''.
\begin{definition}\label{FormalSchDefinition}
Let $E$ be an even--periodic cohomology theory, and let $X$ be a CW--space, so that $X = \colim_\alpha \{X_\alpha\}$ can be written as the colimit over the lattice of inclusions of its compact subspaces $X_\alpha$.  Suppose further that the system of finite $E^*$--algebras $\{E^* X_\alpha\}$ is equivalent to a sub-pro-system of even--concentrated finite $E^*$--algebras $\{E^0 X_{\alpha'} \otimes_{E^0} E^*\}$.  We then set \[X_E := \Spf E^0 X := \{\Spec E^0 X_{\alpha'}\}.\]
\end{definition}

\begin{remark}[{\cite[Section 8.2, Definition 8.15]{StricklandFSFG}}]
A sufficient condition on $X$ to ensure that $X_E$ exists is for $H_* X$ to be torsion--free and even, in which case there is a skeletal structure on $X$ whose subskeleta form a cofinal subsystem of $\{X_\alpha\}$ and which have the desired evenness property.  For instance, this covers the cases of $X = \CP^\infty$ and $X = BU$.  It does not, however, cover $X = BS^1[p^j]$, to which the definition given here nonetheless applies for $E = E_\Gamma$ a Morava $E$--theory.
\end{remark}

Remembering our goal of studying $X$ directly, we would now like to identify a topological system underlying this construction.  The compactness (hence dualizability) of $X_\alpha$ gives \[E^* X_\alpha = \pi_{-*} F(\Susp^\infty_+ X_\alpha, E) = \pi_{-*} \left( F(\Susp^\infty_+ X_\alpha, \S) \sm E \right),\] from which we are inspired to consider the pro-$E_\infty$-ring-spectrum \[\widehat D X = \{F(\Susp^\infty_+ X_\alpha, \S)\}.\]  This pro-spectrum depends only on $X$ and not on $E$, and the system appearing in \Cref{FormalSchDefinition} then arises as the homotopy of the base change \[\eta^* \widehat D X \simeq \{E \sm F(\Susp^\infty_+ X_\alpha, \S)\},\] where $\eta\co \S \to E$ is the unit map for the ring spectrum $E$.  Hence, $\widehat D X$ plays something of a role of a universal object for these constructions.

We now immediately set our eyes on the main Theorem.  Recall the diagram of $A$--modules which defines the cotangent space for a geometric point $\Spec A/\m$ of an affine scheme $\Spec A$:
\begin{center}
\begin{tikzcd}
A \arrow{d} & \m \arrow{l} \arrow{d} & \m^2 \arrow{l} \\
A / \m & \m / \m^2 = T_0^* \Spec A,
\end{tikzcd}
\end{center}
Each angle in the diagram is a cokernel sequence of $A$--modules.  Our goal is to lift this diagram of cokernel sequences of modules to a diagram of cofiber sequences of spectra, such that the original diagram is recovered upon applying $E$--cohomology.  A candidate replacement for $A$ itself is the pro-spectrum $\widehat D X$, and a pointing of the space $X$ induces an augmentation map $\widehat D X \to \S$ to the constant pro-spectrum $\S$.

However, working with $\widehat D X$ is prohibitively complicated: the Spanier--Whitehead duals and pro-system in its definition are designed to mitigate what would be the difference between the cohomology of the limit and the limit of the cohomology of the components, which we may eventually find undesirable and which we would be obligated to thread through our entire discussion.  To avoid this conceptual overhead and potential pitfall, we take inspiration from a corresponding purely algebraic situation: for a field $k$, pro-finite $k$--algebras are equivalent to ind-finite $k$--coalgebras by $k$--linear duality, and \emph{all} $k$-coalgebras are equivalent to ind-finite ones.  Hence, the theory of formal schemes (over $\Spec k$) can be taken to be underpinned by coalgebras instead,\footnote{This perspective is recorded by Demazure~\cite{Demazure} in the case of a field and by Strickland~\cite[Section 4.8]{StricklandFSFG} more generally.} as in the following definition:

\begin{definition}\label{SchDefinition}
A coalgebra $C$ engenders a formal scheme $\Sch C$ via the following dual formula for its functor of points: \[(\Sch C)(T) = \left\{ x \in C \otimes_k T \middle| \begin{array}{c} \Delta(x) = x \otimes x, \\ \eps(x) = 1 \end{array} \right\}.\]
\end{definition}

Returning to spectra, we find the following analogies: pro-finite spectra are equivalent to ind-finite spectra by Spanier--Whitehead duality, and spectra themselves are equivalent to ind-finite spectra.  Hence, the analogous coalgebraic spectral object to the ``profinite ring spectrum'' $\widehat D X$ is $C = \Susp^\infty_+ X$ itself, which we will work with \emph{directly}.  We need only firmly record the sense in which $C$ is a coalgebra before we can start doing coalgebraic geometry with it.

\subsection{Definitions and constructions}

The theory of derived algebraic geometry and spectral schemes is only now coming into view through work of Lurie and others, but the core of \emph{affine} derived algebraic geometry has been laid out for some time now via the theory of structured ring spectra.  One of the first hurdles to clear in this arena is to produce a definition of a ring spectrum (i.e., affine derived scheme) that has an associated $\infty$--category of module spectra (i.e., category of quasicoherent sheaves); such a definition is found in that of an $A_\infty$--ring spectrum.  The precise definition is fairly elaborate, and the reader is better off consulting a textbook for a full account~\cite[Chapter 4]{Lurie}, but the constituent pieces are:
\begin{itemize}
    \item The $\infty$--category $\CatOf{Spectra}$ and the category $\Delta$ of combinatorial simplices.
    \item A \textit{monoidal structure} on $\CatOf{Spectra}$, which can be encoded as a cocartesian fibration $q\co \CatOf{Spectra}^\otimes \to \CatOf{Ass}^\otimes$ of operads.
    \item A particular functor $\mathrm{Cut}\co \Delta^{\op} \to \CatOf{Ass}^\otimes$ to be used to control concatenation of labels for combinatorial simplices.
    \item A functor $A\co \Delta^{\op} \to \CatOf{Spectra}^\otimes$ respecting the pair of functors to $\CatOf{Ass}^\otimes$ and sending ``inert'' morphisms to inert morphisms.  We refer to such morphisms as \textit{$\Delta^{\op}$--algebras in $\CatOf{Spectra}$}.
\end{itemize}
Collectively, these amount to a choice of labelings of the combinatorial simplices, identifications $A[n] \simeq A[1]^{\sm n}$, a unit map $A[0] \to A[1]$, a multiplication $A[2] \to A[1]$, and a sea of compatibilities encoding associativity, unitality, and $\sm$--composability.

Since we are interested in coassocative coalgebras rather than associative algebras, we tweak this definition only very slightly:
\begin{definition}[{cf.\ \cite[Definition 4.12]{Haugseng}, \cite[Proposition 4.1.2.15]{Lurie}}]
A \textit{coassociative coalgebra} in a monoidal $\infty$--category $\CatOf C$ is a $\Delta^{\op}$--algebra in $\CatOf C^{\op}$, i.e., an associative algebra in the opposite category $\CatOf C^{\op}$.
\end{definition}

\begin{lemma}\label{SpacesAreCoalgebras}
A pointed space $X$ determines an associative algebra object in the monoidal $\infty$--category $\CatOf{Spectra}^{\op}$ and in its homological Bousfield localization $\CatOf{Spectra}_E^{\op}$ for a homology theory $E$.
\end{lemma}
\begin{proof}
Pointed spaces (using the underlying diagonal map of sets) already determine strictly coassociative algebra objects on the $1$--categorical level.  By localizing away from the weak equivalences and passing to the associated $\infty$--category, we recover an associative algebra object in $\CatOf{Spaces}_*^{\op}$.  Because the stabilization functor \[\Sigma^\infty_+\co \CatOf{Spaces}_* \to \CatOf{Spectra}\] respects the Cartesian monoidal structure on $\CatOf{Spaces}_*$ and the smash monoidal structure on $\CatOf{Spectra}$~\cite[Propositions 4.8.2.9 and 4.8.2.18]{Lurie} and because the monoidal structure on $\Gamma$--local spectra is defined so that the localization functor $L_\Gamma$ is monoidal, it follows that $L_\Gamma \Susp^\infty_+ X$ is a coaugmented coassociative coalgebra in $\CatOf{Spectra}_\Gamma$.
\end{proof}

Our spectral analogue of the first fiber sequence in the diagram defining the algebraic tangent space is thus given by the following diagram:
\begin{center}
\begin{tikzcd}
C \arrow{r} & M \arrow{r} & ? \\
\S \arrow["\eta"]{u} & ? \arrow{u}
\end{tikzcd}
\end{center}
where $C = \Susp^\infty_+ X$ is a coalgebraic spectrum and $\eta\co \S \to C$ is the coaugmentation map.  In order to construct the second fiber sequence, we would like to interpret this diagram as a sequence of $C$--comodules, and we would like access to a coalgebraic analogue of the product of ideals.  Here, again, we can lean on the already-developed theory of modules for associative algebra spectra: there is a \textit{nonsymmetric $\infty$--operad} $\Delta^{\op}_{/ [1]}$~\cite[Definition 4.2]{Haugseng} whose algebras in $\CatOf C$ amount to $\CatOf C$--labeled combinatorial simplices with one edge labeled by a fixed spectrum $M$ (the bimodule), all edges indexed below it labeled by a fixed spectrum $A$ (the left algebra), all edges indexed above it labeled by a fixed spectrum $B$ (the right algebra), together with maps describing the associative algebra structures of $A$ and of $B$, the left-action of $A$ on $M$, the right-action of $B$ on $M$, and compatibilities encoding associativity, unitality, and $\sm$--composability.

\begin{definition}[{\cite[Definition 4.12]{Haugseng}, \cite[Definition 4.3.1.6]{Lurie}}]
For a coassociative coalgebra $C$ in a monoidal $\infty$--category $\CatOf C$, the category of $C$--comodules is the subcategory of bimodules in $\CatOf{C}^{\op}$ whose left- and right-algebras are equivalent to $C$.
\end{definition}

Even with this technology, a product of ideals is too much to ask: the missing theory of ideals for structured ring spectra remains, to date, a thorn in homotopy theorists's sides.\footnote{The present situation here is a bit complicated.  It is known that a ring spectrum can be localized away from any homotopy element, but that quotienting some ring spectra by certain homotopy elements can yield non-associative or even non-unital spectra (e.g., respectively, $\S / p$ for odd primes $p$ and $\S / 2$).  Meanwhile, the intended collection of ideals of $\S$ is considered well-known: the theory of ideals at the level of perfect $\S$--modules is understood through the nilpotence and periodicity theorems, and the reflection of those results in the homotopy elements of $\S$ is extremely indirect.  However, there is an $\S$--algebra $MU$ which does, in a certain stacky sense, reveal this ideal structure at the level of homotopy elements (but which still does not admit all the ring spectrum quotients that one might naively expect).  At any rate, the reader must temper their expectations about the theory being developed here, as it serves our immediate purposes but not much more.}  However, in the classical case that two $A$--ideals $I$ and $J$ are \emph{principal} with generators that have \emph{no annihilators}, there is a natural isomorphism $I \otimes_A J \cong I \cdot J$.  This includes our primary case of interest, where $A = k\ps{x}$ and $I = J = \m = (x)$ is the ideal of functions vanishing at $x = 0$.\footnote{The augmentation ideal of a higher--dimensional power series algebra is \emph{not} principal.}  Motivated by this observation, we, again, borrow a definition from Haugseng and Lurie:

\begin{definition}[{\cite[Definition 4.13, Lemma 4.14]{Haugseng}, cf.\ \cite[Definition 4.4.1.1, Proposition 4.4.1.11]{Lurie}}]
There is a nonsymmetric $\infty$--operad $\Lambda^{\op}_{/[n]}$ which arises as the gluing of $n$ face maps in from $\Delta^{\op}_{/[1]}$ along $(n-1)$ shared instances of $\Delta^{\op} = \Delta^{\op}_{/[0]}$ in a chain.  Accordingly, a $\Lambda^{\op}_{/[n]}$--algebra in $\CatOf C^{\op}$ selects a family of $C_{i-1}$--$C_i$--cobimodule objects $M_i$ for $1 \le i \le n$: this is a \textit{length $n$ chain of compatible comodules}.  A \textit{length $n$ cotensor witness} is then a $\Delta^{\op}_{/[n]}$--algebra in $\CatOf C^{\op}$:
\begin{center}
\begin{tikzcd}
\Lambda^{\op}_{/[n]} \arrow{d} \arrow{rd} \\
\Delta^{\op}_{/[n]} \arrow[densely dotted]{r} & \CatOf C.
\end{tikzcd}
\end{center}
In particular, $\Delta^{\op}_{/[n]}$ has $\binom{n}{2}$ face maps $\Delta^{\op}_{/[1]} \to \Delta^{\op}_{/[n]}$, i.e., a $C_i$--$C_j$--cobimodule for every choice of $i$ and $j$.  A \textit{cotensor product} of the chain is the $C_0$--$C_n$--cobimodule extracted from a witness in this way.
\end{definition}

\begin{lemma}[{\cite[Lemma 4.19, Corollary 4.20]{Haugseng}, cf.\ \cite[Theorem 4.4.2.8]{Lurie}}]\label{CotensorExistenceLemma}
Let $P\co \Lambda^{\op}_{/[n]} \to \CatOf C^{\op}$ be a length $n$ chain of compatible cobimodules in $\CatOf C$, and let $F\co \Delta^{\op}_{/[n]} \to \CatOf C^{\op}$ be a filling of $P$ to a length $n$ tensor witness.  The spectrum underlying the cotensor product comodule witnessed by $F$ is necessarily weakly equivalent to the limit of the cobar construction on the compatible cobimodules.  Conversely, suppose that $n = 2$ and only $P$ is given.  If the natural $\CatOf C^{\op}$--map \[|X \otimes B(M; S; N) \otimes Y| \to X \otimes |B(M; S; N)| \otimes Y\] is an equivalence for all objects $X$ and $Y$, then there must exist a filler $F$. \qed
\end{lemma}

\begin{remark}
Haugseng's filling result is not specific to cotensor products of length $2$, and it is functorial as the $\Lambda^{\op}_{/[n]}$--algebra varies.  However, we will be content with the weakened result stated here.
\end{remark}

The primary obstacle for us now is the natural equivalence hypothesis in the Lemma.  In our setting of $\CatOf C = \CatOf{Spectra}$ or $\CatOf C = \CatOf{Spectra}_\Gamma$, totalizations of cosimplicial objects \emph{do not} typically commute with the (local) smash product of (local) spectra.  So, while we can use the cobar construction to define a cotensor product and various maps concerning it, essentially no good properties can be deduced for the construction without first manually checking this condition about commuting limits and smash products.  For the moment, we content ourselves with constructing the maps, leaving the question of properties for the next section.

\begin{lemma}
For a right $C$--comodule $M$, there is a natural retraction \[M \to M \cotensor_C C \to M,\] and dually for a left $C$--comodule $N$. \qed
\end{lemma}

\begin{lemma}\label{DeltaFactorsThroughCoideals}
Let $M$ be a $C$--$C$--cobimodule equipped with a map $\pi\co C \to M$ of cobimodules.  There is then a diagram
\begin{center}
\begin{tikzcd}
C \arrow["\pi"]{r} \arrow["\Delta_C"]{d} & M \arrow["\psi_R"]{d} \arrow["\widetilde \Delta_R"]{rd} \\
C \cotensor_C C \arrow["\pi \cotensor_C C"]{r} \arrow{d} & M \cotensor_C C \arrow["M \cotensor_C \pi"]{r} \arrow{d} & M \cotensor_C M \\
C \arrow{r} & M
\end{tikzcd}
\end{center}
where the vertical sequences are retractions.  Alternatively, there is also a map $\widetilde \Delta_L$ using the left $C$--comodule structure of $M$.
\end{lemma}
\begin{proof}
The arrow $\Delta_C$ could just as well be called $\psi_C$, and by this name the top square commutes because $\pi$ is a homomorphism of cobimodules.  The arrow $\widetilde \Delta_R$ is then defined by either composite.
\end{proof}

Classically, if $\pi$ is surjective, then $\Delta_R$ and $\Delta_L$ will agree on points exactly if $M$ is taken to be a coideal.  The word ``coideal'' is too bold in our homotopical context, since it suggests that the quotient is again a coalgebra.  We instead adopt the following terminology:

\begin{definition}
Let $M$ be a $C$--$C$--cobimodule spectrum equipped with a cobimodule map $\pi\co C \to M$.  Such a cobimodule spectrum is a \textit{symmetric cobimodule} (for $C$) when $\widetilde \Delta_R$ and $\widetilde \Delta_L$ are homotopic, in which case we will use $\widetilde \Delta$ as an abbreviation for either.\footnote{This is \emph{not} the robust definition that a party setting out to build up a theory of spectral coalgebraic algebra would choose.  They would probably prefer to include coherence conditions on this symmetry.  We, however, are only setting out to support \Cref{CotangentSpectrumDefn}.}
\end{definition}

\begin{lemma}
Let $f\co D \to C$ be a map of coalgebras, and suppose that there is a splitting of the cofiber:
\begin{center}
\begin{tikzcd}
D \arrow["f"]{r} & C \arrow["\pi"]{r} & \cofib f \arrow[bend right, densely dotted]{l}.
\end{tikzcd}
\end{center}
The cofiber of $f$, considered as a $C$--$C$--cobimodule, is then a symmetric cobimodule.
\end{lemma}
\begin{proof}
We enlarge the diagram of \Cref{DeltaFactorsThroughCoideals} to include $\widetilde \Delta_L$ and $\widetilde \Delta_R$:
\begin{center}
\begin{tikzcd}[row sep=2em]
D \arrow["f"]{r} & C \arrow["\pi"]{rr} \arrow["\Delta"]{dd} & & \cofib f \arrow[bend right, densely dotted]{ll} \arrow["\psi_R"]{dl} \arrow[bend left, "\widetilde \Delta_L"]{d} \arrow[bend right, "\widetilde \Delta_R"']{d} \arrow["\psi_L"]{dr} \\
& & \cofib f \sm C \arrow["1 \sm \pi"]{r} & \cofib f \sm \cofib f & C \sm \cofib f \arrow["\pi \sm 1"]{l} \\
& C \sm C \arrow["\pi \sm 1"]{ru} \arrow["\pi \sm \pi"]{rru} \arrow["1 \sm \pi", bend right=10]{rrru}.
\end{tikzcd}
\end{center}
We see that the two maps $\tilde \Delta_L$ and $\tilde \Delta_R$ agree upon precomposition to $C$.  Since $\cofib f$ splits off of $C$, they agree on $\cofib f$ as well.
\end{proof}

\begin{lemma}
For $M$ a symmetric cobimodule admitting cotensor powers $M^{\cotensor_C k}$ and $M^{\cotensor_C (k+1)}$, all of the following maps are homotopic:
\[M^{\cotensor_C k} \xrightarrow{1 \cotensor \cdots \cotensor 1 \cotensor \tilde \Delta \cotensor 1 \cotensor \cdots \cotensor 1} M^{\cotensor_C (k+1)}.\]
\end{lemma}
\begin{proof}
The definition of the cobar construction gives a homotopy between the following two maps:
\begin{center}
\begin{tikzcd}
M \cotensor_C M \arrow[shift left=0.5\baselineskip, "\psi_R \cotensor_C 1"]{r} \arrow[shift right=0.5\baselineskip, "1 \cotensor_C \psi_L"]{r} & M \cotensor_C C \cotensor_C M.
\end{tikzcd}
\end{center}
Using the symmetric cobimodule property of $M$, we can trade $\tilde \Delta_L$ for $\tilde \Delta_R$, and hence we can transfer the $\tilde \Delta$ to any coordinate.
\end{proof}

We are now poised to properly state our construction:
\begin{definition}\label{CotangentSpectrumDefn}
Given a pointed coalgebra spectrum $\eta\co \S \to C$, we define $C\range 11$ by the following chain of fiber sequences:
\begin{center}
\begin{tikzcd}
C \arrow{r} & M \arrow["\widetilde \Delta"]{r} & M \cotensor_C M \\
\S \arrow["\eta"]{u} & C\range 11 \arrow{u}.
\end{tikzcd}
\end{center}
\end{definition}

\begin{remark}
If $C$ is merely a pointed coalgebra spectrum, this is as much of the diagram as we can form at this time.  After all, we have not yet shown that $M \cotensor_C M$ is a $C$--$C$--cobimodule, and hence we cannot guarantee the existence of ``$M \cotensor_C (M \cotensor_C M)$''.  However, if we start with a space $X$ and follow \Cref{SpacesAreCoalgebras}, then we can form the tower of iterated cobar diagrams in the $1$--category of spaces:
\begin{center}
\begin{tikzcd}
X_+ \arrow{r} & X \arrow{r} & \Omega(X; X_+; X) \arrow{r} & \Omega(X; X_+; X; X_+; X) \arrow{r} & \cdots.
\end{tikzcd}
\end{center}
Pushing these diagrams forward along $\Susp^\infty$ and along the localizer $L_\Gamma$ then produces inverse diagrams of ($\Gamma$--local) spectra.  Thus, we are at least assured that, for $C = \Susp^\infty_+ X$, the diagrams \emph{determining} the higher cotensor powers (and hence the annular tower) are well-defined.  We are still obligated to justify their utility by computing the value of $E_\Gamma$ on their limits.
\end{remark}






\subsection{Computational tools}\label{ComputationalToolsSection}

We would like to justify \Cref{CotangentSpectrumDefn} above by calculating $E_\Gamma^0 C\range 11$ and checking that it gives the desired cotangent module, and then we would also like to extend it to the full tower presented in the main Theorem.  In view of \Cref{CotensorExistenceLemma}, this rests directly on having tools available to compute the homology and cohomology of a cotensor product of comodule spectra.  With a computational task ahead, it is now convenient to introduce Morava $K$--theory.

\begin{definition}
The coefficient ring $\pi_0 E_\Gamma$ is a power series ring over $\W_p(k)$, itself a complete local ring with maximal ideal generated by $p$.  The \textit{Morava $K$--theory spectrum}, $K_\Gamma$, is the quotient of $E_\Gamma$ in $E_\Gamma$--modules by a generating regular sequence of the maximal ideal of $\pi_0 E_\Gamma$.
\end{definition}

\noindent As a result, $\pi_* K_\Gamma = k[u^\pm]$ is a graded field.  This makes it extremely valuable for computations---especially ones such as ours, where large smash products arise, since $K_\Gamma$ has K\"unneth isomorphisms.  Remarkably, despite being defined as a quotient, ``$K_\Gamma$ carries the same information as $E_\Gamma$'' in the precise sense that $K_\Gamma$--acyclics, homological or cohomological, agree with cohomological $E_\Gamma$--acyclics~\cite[Proposition 2.5]{HoveyStrickland}.  Hence, any of $E_\Gamma$--cohomology, $K_\Gamma$--cohomology, or $K_\Gamma$--homology are equally good for testing $\Gamma$--local equivalences.\footnote{Remarkably, $E_\Gamma$--homology does not belong to this list, essentially for the same reason that $\Z_p \otimes_{\Z} \Z_p$ is not $p$--complete.  We will address this further in \Cref{DefnCtsMoravaEThy}.}  We will favor the last option, both for technical reasons (cf.\ \cite{BarthelBeaudryPeterson} and \Cref{LimitAppendix}) and because it covariantly converts the coalgebraic spectrum $C$ into the $(K_\Gamma)_*$--coalgebra $(K_\Gamma)_* C$, then into the coalgebraic formal scheme $\Sch (K_\Gamma)_* C$.  To save on notational overhead, we will often write ``$K$'' alone when some fixed $\Gamma$ is understood.

We now turn back to the task at hand: given $C$--$C$--cobimodules $M$ and $N$, we want to study the natural map \[X \sm \lim_{\CatOf C} \Omega(M; C; N) \sm Y \to \lim_{\CatOf C} \left( X \sm \Omega(M; C; N) \sm Y \right)\] and to find conditions under which it becomes a $K$--homology equivalence.  The presentation of $M \cotensor N$ as the totalization of a cosimplicial object equips it with a coskeletal filtration and hence a spectral sequence approximating its homology.  We might hope to use this to gain a foothold on our problem, but the homology of an inverse limit of spectra typically compares poorly with the inverse limit of the homologies of the individual spectra in the system.  Remarkably, a result of Sadofsky shows that $K$--homology does not suffer too badly from this.

\begin{theorem}[{cf.\ \Cref{SadofskyInverseLimits}, \cite{BarthelBeaudryPeterson}}]\label{SadofskysTheoremInCotangentSection}
Take $p \gg \height \Gamma$, and let $\{X_\alpha\}_\alpha$ be a pro-system of $\Gamma$--local spectra such that $\{K_* X_\alpha\}_\alpha$ is a Mittag-Leffler system of $K_*$--modules.  There is then a convergent spectral sequence of signature \[R^* \mylim{\CatOf{Comod}_{K_* K}}\alpha \{K_* X_\alpha\}_\alpha \Rightarrow K_* \mylim{\CatOf{Spectra}_\Gamma}\alpha \{X_\alpha\}_\alpha,\] where the right--derived inverse limit forming the input to the spectral sequence is taken in the category of $K_* K$--comodules. \qed
\end{theorem}

In order to make use of this spectral sequence, we need to compute some of its inputs.  The homologies of the finite stages of the filtration are accessible because homology does pass through \emph{finite} limits, and these admit the following uniform description.

\begin{theorem}\label{NaiveCotorSSeq}
Let $C$ be a coalgebra spectrum, $M$ a right $C$--comodule spectrum, and $N$ a left $C$--comodule spectrum.  Writing $F$ for the fiber of the counit map $\eps\co C \to \S$, there is an $n$--indexed system of spectral sequences: \[\left. \begin{array}{c} E^1_{*, *} \cong K_* M \otimes_{K_*} (K_* F)^{\otimes (* \le n)} \otimes_{K_*} K_* N, \\ E^2_{*, (* < n)} \cong \Cotor^{K_* C}_{*, *}(K_* M, K_* N)\end{array} \right\} \Rightarrow K_* \Tot_n \Cobar(M; C; N)\] whose inverse limit in the category of $K_*$--modules is the spectral sequence \[\left. \begin{array}{c} E^1_{*, *} \cong K_* M \otimes_{K_*} (K_* F)^{\otimes *} \otimes_{K_*} K_* N, \\ E^2_{*, *} \cong \Cotor^{K_* C}_{*, *}(K_* M, K_* N)\end{array} \right\} \Rightarrow \mylim{\CatOf{Modules}_{K_*}}{n} K_* \Tot_n \Cobar(M; C; N).\]  The spectral sequences in the system are strongly convergent, and the full spectral sequence is conditionally convergent.
\end{theorem}
\begin{proof}
This is an aggregation of several standard results in the literature.  Ravenel and Wilson~\cite[Section 2]{RavenelWilson} provide a convenient summary of the bar spectral sequence, and these spectral sequences arise as its dual.  Ravenel~\cite[Appendix A1]{RavenelGreenBook} also provides a collection of results on the homological algebra of comodules and in particular gives a definition and lists properties for ``$\Cotor$''.  Finally, Boardman~\cite[Theorem 7.1]{Boardman} provides tools for analyzing the convergence of the spectral sequences.
\end{proof}

\noindent We now sew these together to analyze Sadofsky's inverse limit spectral sequence in the case at hand.  This is mostly an exercise in elementary homological algebra, so requires substantial bookkeeping but is otherwise straightforward.

\begin{theorem}\label{CorrectConvergence}
Continue to assume $p \gg \height \Gamma$.  Assume that the natural map \[K_s \Tot_t \Cobar(M; C; N) \to K_s \Tot_{t-1} \Cobar(M; C; N)\] is injective when $s+t$ is odd and surjective when $s+t$ is even.\footnote{In particular, this is satisfied if $M$, $N$, and $C$ have even--concentrated $K$--homology.}  There then is an additional convergent spectral sequence \[R^* \mylim{\CatOf{Comod}_{K_* K}}{t} \{K_* \Tot_t \Cobar(M; C; N)\}_t \Rightarrow K_* (M \cotensor_C N),\] where the derived inverse limit is taken in the category of $K_* K$--comodules.
\end{theorem}
\begin{proof}
Throughout, we will consider the degree $s+t$ part of the $t$-cochains $C^t_{s+t}$, coboundaries $B^t_{s+t}$, and cocycles $Z^t_{s+t}$ of the normalized cobar complex \[\Cobar(K_* M; K_* C; K_* N).\]  With an induction beginning at $t = 0$ and $t = 1$, we claim the $K$-homology of the $t${\th} level of the tower (i.e., the target of the $t${\th} partial spectral sequence) is
\[
K_s \Tot_t \Cobar(M; C; N) = D^1_{s, t} = 
\begin{cases}
D^1_{s, t-1} \oplus \frac{C^t_{\frac{s+t}{2}}}{B^t_{\frac{s+t}{2}}} & \hbox{if $s+t$ is even,} \\
D^1_{s, t-2} \oplus H^{t-1}_{\frac{s+t-1}{2}} & \hbox{if $s+t$ is odd,}
\end{cases}
\]
where $D^1_{*, *}$ denotes the rear of the exact couple of the full spectral sequence.  In particular, this formula will show that the inverse limit tower is Mittag-Leffler on $K$--homology, so that Sadofsky's hypotheses are satisfied.  The theorem then follows.

Induction proceeds by considering one triangle in that exact couple:
\begin{center}
\begin{tikzcd}[column sep=1em]
K_s \Tot_{t-1} \Cobar(M; C; N) & K_s \Tot_t \Cobar(M; C; N) \arrow{l} \\
& K_s \Loops^t(M \wedge C^{\wedge t} \wedge N) \arrow{u} \arrow[leftarrow]{lu}{[-1]}
\end{tikzcd}
=
\begin{tikzcd}[column sep=1em]
D^1_{s, t-1} & D^1_{s, t} \arrow{l} \\
& E^1_{s, t} \arrow{u} \arrow[leftarrow]{lu}{[-1]}.
\end{tikzcd}
\end{center}
The bottom vertex in the triangle is the $K_*$-module of cochains, and we take $s$ to be a degree in which $K_{s+t}(M \wedge C \wedge \cdots \wedge C \wedge N)$ is nonvanishing, i.e., $s + t$ is even.  Then, the triangle unrolls into an exact sequence which, using the inductive hypothesis, takes the form
\begin{center}
\begin{tikzcd}[column sep=1.1em]
0 \arrow{r} & D^1_{s+1,t} \arrow{r} \arrow[double,-]{d} & D^1_{s+1,t-1} \arrow{r} \arrow[double,-]{d} & E^1_{s,t} \arrow{r} \arrow[double,-]{d} & D^1_{s,t} \arrow{r} \arrow[double,-]{d} & D^1_{s,t-1} \arrow{r} \arrow[double,-]{d} & 0 \\
0 \arrow{r} & H^{t-1}_{\frac{s+t}{2}} \oplus X \arrow{r}{i \oplus 1} & \frac{C^{t-1}_{\frac{s+t}{2}}}{B^{t-1}_{\frac{s+t}{2}}} \oplus X \arrow{r}{\partial \oplus 0} & C^t_{\frac{s+t}{2}} \arrow{r}{\pi \oplus 0} & \frac{C^t_{\frac{s+t}{2}}}{B^t_{\frac{s+t}{2}}} \oplus Y \arrow{r}{0 \oplus 1} & Y \arrow{r} & 0
\end{tikzcd}
\end{center}
for some modules $X$ and $Y$ to be determined.  We then splice three of these long sequences together to form \Cref{FigThreeSequences}, which is labeled in terms of the cobar cohomology groups in \Cref{FigCobarSequences}.  The action lies in the zig-zag containing $D^1_{s, t+1}$ through $D^1_{s, t-2}$: if $D^1_{s, t-1}$ and $D^1_{s, t-2}$ are assumed to be as claimed, then the observation involving $Y$ shows that $D^1_{s, t-1}$ is as well, which in turn determines $D^1_{s, t+1}$ as the kernel of the map $D^1_{s, t} \to E^1_{s-1,t+1}$, i.e., as the cocycle subgroup of the cochain group, taken modulo the coboundaries.  This proves the inductive claim.
\end{proof}

\begin{figure}
\begin{center}
\begin{tikzcd}[row sep=0.8em, ampersand replacement=\&]
\& \& \& \& 0 \arrow{d} \\
\& \& \& \& D^1_{s, t+1} \arrow{d} \& 0 \arrow{d} \\
0 \arrow{r} \& D^1_{s+1,t} \arrow{r} \& D^1_{s+1,t-1} \arrow{r} \& E^1_{s,t} \arrow{r} \& D^1_{s,t} \arrow{r} \arrow[bend left]{dl} \& D^1_{s, t-1} \arrow{r} \arrow{d} \& 0 \\
0 \& D^1_{s-1,t} \arrow{l} \& D^1_{s-1,t+1} \arrow{l} \& E^1_{s-1,t+1} \arrow{l} \& \& D^1_{s,t-2} \arrow[bend left]{dl} \\
\& 0 \& D^1_{s-1, t-2} \arrow{l} \& D^1_{s-1,t-1} \arrow{l} \& E^1_{s-1,t-1} \arrow{l}
\end{tikzcd}
\end{center}
\caption{Three interacting exact sequences.}\label{FigThreeSequences}

\begin{center}
\begin{tikzcd}[column sep=0.65cm, row sep=0.8em, ampersand replacement=\&]
\& \& \& \& 0 \arrow{d} \\
\& \& \& \& H^t_{\frac{s+t}{2}} \oplus D^1_{s,t-1} \arrow{d} \& 0 \arrow{d} \\
0 \arrow{r} \& D^1_{s+1,t} \arrow{r} \& D^1_{s+1,t-1} \arrow{r} \& C^t_{\frac{s+t}{2}} \arrow{r} \& \frac{C^t_{\frac{s+t}{2}}}{B^t_{\frac{s+t}{2}}} \oplus D^1_{s,t-1} \arrow{r} \arrow[bend left]{dl} \& D^1_{s, t-1} \arrow{r} \arrow{d} \& 0. \\
0 \& D^1_{s-1,t} \arrow{l} \& D^1_{s-1,t+1} \arrow{l} \& C^{t+1}_{\frac{s+t}{2}} \arrow{l} \& \& D^1_{s,t-2} \arrow[bend left]{dl} \\
\& 0 \& D^1_{s-1, t-2} \arrow{l} \& D^1_{s-1,t-1} \arrow{l} \& E^1_{s-1,t-1} \arrow{l}
\end{tikzcd}
\end{center}
\caption{The interacting sequences with cobar complex names.}\label{FigCobarSequences}
\end{figure}

\subsection{Computation in the case of a formal curve}

In general, it is very hard to control the spectral sequence of \Cref{CorrectConvergence}.  Even computing these derived functors is prohibitively complicated in almost any nondegenerate case---for instance, Hopkins~(\cite[Section 14]{MITETheory},~\cite{BarthelBeaudryPeterson}) has recommended them in an approach to analyzing the chromatic splitting conjecture.\footnote{Computing these derived functors can be compared to computing certain local cohomology groups and to certain ``noncontinuous'' forms of the group cohomology of the Morava stabilizer group.  These are both very difficult invariants.}  This moves us to consider progressively more specialized situations where we can fully determine this spectral sequence, beginning with the case that the coalgebraic formal scheme $\Sch K_0 C$ (cf. \Cref{SchDefinition}) is a formal variety.

\begin{definition}
A $k$--coalgebra $C$ over a field $k$ will be called a \textit{formal variety (of dimension $n$)} when, for any augmented nilpotent $k$--algebra $A$ with maximal ideal $\m$, there is a natural isomorphism between the group--like elements of $C \otimes A$ and $\m^{\times n}$.\footnote{Equivalently, the dual profinite algebra $C^*$ is isomorphic to a power series ring.}  A spectral coalgebra $C$ will be called a \textit{formal coalgebraic variety spectrum (local to $\Gamma$)} when $K_* C$ is even--concentrated and $K_0 C$ is a formal variety in coalgebras.  In both settings, we use \textit{formal curve} as a synonym for a formal variety of dimension $1$.
\end{definition}

For the rest of this section, we take $\eta\co \S \to C$ to be a pointed coalgebraic formal \emph{curve} spectrum local to $\Gamma$, and we write $M$ for the cofiber of $\eta$.  Supposing that $M^{\cotensor_C j}$ has been shown to support the structure of an even $C$--$C$--cobimodule spectrum, we now inductively pursue an even $C$--$C$--cobimodule structure on $M^{\cotensor_C (j+1)}$.

\begin{theorem}\label{CoskelTowerIsProconstant}
The system $\{K_*^n \Cobar(M; C; M^{\cotensor_C j})\}_n$ is pro-constant.
\end{theorem}
\begin{proof}
This system appears as the rear of the exact couple in \Cref{CorrectConvergence}, where we showed that it was presented levelwise as direct sums of two sorts of groups: $\Cotor$ groups and cochain groups modulo coboundaries.  These groups are then linked together by appropriate projections (away from a direct sum factor) and inclusions (of cocycles modulo coboundaries---i.e., $\Cotor$ groups---into cochains modulo coboundaries).  Our conclusion will follow from a direct calculation of these $\Cotor$ groups.

This problem may be addressed by standard tools in homological algebra and in local cohomology, which is perhaps more visible after taking linear duals and passing to profinite algebras: \[\Cotor_{K_* C}^{*, *}(K_* M, K_* N) \cong \Tor_{K_* C^\vee}^{*, *}(K_* M^\vee, K_* N^\vee)^\vee.\]  Under our hypotheses, we may choose isomorphisms \[K_* C^\vee \cong K_* \ps{x}, \quad K_* M^\vee = \<x\>.\]  This shows $(K_* M)^\vee$ to be a free $(K_* C)^\vee$--module, hence the higher $\Tor$ groups vanish.\footnote{For $C$ of higher dimension, these $\Cotor$ groups are bounded above, and so we may draw a similar conclusion.  However, they fail to be even, and this harms the induction irreparably.}

We now return to the original problem of pro-constancy.  Because so many $\Cotor$ groups themselves vanish, we find that the maps in the pro-system are zero except on those nonvanishing factors, where the maps are instead the respective identity morphisms.  It follows that, for each fixed $s$, the tower $D^1_{s, t}$ is pro-constant.
\end{proof}

\begin{corollary}\label{CotensorInterchangeLaw}
There is an isomorphism \[K_*(M \cotensor_C M^{\cotensor_C j}) \cong K_* M^{\cotensor_{K_* C} (j+1)}.\]
\end{corollary}
\begin{proof}
Because the tower is pro-constant, the higher derived inverse limit functors of \Cref{CorrectConvergence} vanish, and hence we need only contend with the higher derived functors of cotensoring itself.  These also vanish, again because the maximal ideal in a one--dimensional power series algebra is a free module.
\end{proof}

\begin{corollary}
The spectrum $M^{\cotensor_C (j+1)}$ is a $C$--$C$--cobimodule spectrum.
\end{corollary}
\begin{proof}
Tensoring the pro-constant towers of \Cref{CoskelTowerIsProconstant} with $K_* X$ and $K_* Y$ does not disturb their pro-constancy.  It follows that the natural map \[K_* X \otimes K_* M^{\cotensor_{K_* C} (j+1)} \otimes K_* Y \to K_* \mylim{\CatOf{Spectra}_\Gamma}{n} \Tot_n (X \sm \Cobar(M; C; M^{\cotensor_C j}) \sm Y)\] is an isomorphism, and hence the conditions of \Cref{CotensorExistenceLemma} are satisfied.
\end{proof}

This completes the induction and shows that $M^{\cotensor_C j}$ is a $C$--$C$--cobimodule spectrum for all values of $j$.  We use this to justify the following definition.

\begin{definition}\label{AnnularTowerDefn}
Writing $C\range n\infty$ for $M^{\cotensor_C n}$, there are natural projection (or ``pinch'') maps $C\range n\infty \to C\range{n'+1}\infty$ for $n' \ge n$.  We define spectra $C\range n{n'}$ by the fiber sequence \[C\range n{n'} \to C\range n\infty \to C\range {n'+1}\infty\] and the \textit{annular tower} is the sequential system
\begin{center}
\begin{tikzcd}
C\range0\infty \arrow{r} & C\range1\infty \arrow{r} & \cdots \arrow{r} & C\range n\infty \arrow{r} & C\range{n+1}\infty \arrow{r} & \cdots \\
C\range00 \arrow{u} & C\range11 \arrow{u} & \cdots & C\range nn \arrow{u} & C\range{n+1}{n+1} \arrow{u} & \cdots.
\end{tikzcd}
\end{center}
\end{definition}

Our computational tools have been explicit enough that we can now also deduce that $K$--cohomology takes the correct value on $C\range 11$.  By \Cref{CotensorInterchangeLaw}, we have that $K^0 C\range 11$ indeed belongs to the exact sequence
\begin{center}
\begin{tikzcd}
K^0 C\range 11 & K^0 C\range 1\infty \arrow{l} & K^0 C\range 1\infty \otimes_{K^0 C} K^0 C\range 1\infty \arrow{l} \arrow[equal]{d} \\
& & K^0 C\range 1\infty \cdot K^0 C\range 1\infty \arrow["\mu"]{ul}.
\end{tikzcd}
\end{center}

\begin{definition}\label{TangentSpectrumOfCoalgebraicCurve}
In the case where $C$ is a coalgebraic formal curve spectrum, we write $T_\eta C$ for the spectrum $C\range 11$, so that the isomorphism of \Cref{TangentSpectrumOfCoalgebraicCurve} reads as an interchange law \[K^* T_\eta C = T^*_{\eta_K} C_K.\]
\end{definition}

\subsection{Koszul duality}
\label{KoszulDualitySection}

We now give an alternative interpretation of the annular tower,\footnote{The ideas in this subsection will not recur, and an uninterested reader may safely skip ahead.} which we first motivate by example.  In \Cref{ExampleTCPinfty} below, we will show for $C = \Susp^\infty_+ \CP^\infty$ that $T_\eta C \simeq \S^2 \simeq \Susp^\infty \CP^1$ and that the annular decomposition of $C$ is into one sphere of each even dimension.  This familiar cellular filtration of $\CP^\infty$ also arises from a second avenue that, at a glance, appears completely separate.  By appealing to the equivalence \[\CP^\infty \simeq BU(1),\] we may use the bar filtration on $BU(1)$ to produce a filtration on $\CP^\infty$ which has filtration quotients $(\Susp U(1))^{\sm n} \simeq \S^{2n}$.  Moreover, there is a tool for mechanically recognizing this phenomenon: the input $U(1)$ to the functor $B$ may be recovered as the based loopspace \[U(1) \simeq \Loops \CP^\infty,\] which is a specific example of the familiar equivalence of categories
\begin{center}
\begin{tikzcd}
\CatOf{ConnectedSpaces} \arrow[shift left=0.25\baselineskip, "\Loops"]{r} & \CatOf{Grouplike}A_\infty\CatOf{Spaces} \arrow[shift left=0.25\baselineskip, "B"]{l} .
\end{tikzcd}
\end{center}
Koszul duality (or ``co/bar duality'') captures this phenomenon in maximum generality, and it turns out the annular tower may be viewed as a specific instance of it: any coalgebraic formal curve spectrum $C$ can be written as ``$BG$'' for an $A_\infty$ ring spectrum $G$ extracted from the bottom annular layer $T_\eta C$.

Our starting point is a second presentation of $T_\eta C$, arising as follows:

\begin{lemma}\label{CotensoringPreservesFibers}
Given a fiber sequence of $C$--$C$--cobimodule spectra $M'' \to M \to M'$ and a fourth cobimodule spectrum $N$, the sequence \[M'' \cotensor_C N \to M \cotensor_C N \to M' \cotensor_C N\] is also a fiber sequence.
\end{lemma}
\begin{proof}
Because $- \cotensor_C N$ is defined by taking the inverse limit of a diagram constructed from smashing the input with the fixed spectra $C^{\sm n} \sm N$, this functor preserves fiber sequences.
\end{proof}

\begin{corollary}\label{TangentSpectrumInOneCotor}
There is a natural equivalence $\S \cotensor_C M \simeq C \range 11$.
\end{corollary}
\begin{proof}
By applying \Cref{CotensoringPreservesFibers}, one can construct the annular tower by iteratively cotensoring the fiber sequence
\begin{center}
\begin{tikzcd}
C \arrow{r} & M \\
\S \arrow["\eta"]{u}
\end{tikzcd}
\end{center}
with cotensor powers of $M$ and sewing together the overlapping nodes:
\begin{center}
\begin{tikzcd}[column sep=1em]
C \arrow{r} & M \arrow[equal]{r} & M \cotensor_C C \arrow{r} & M \cotensor_C M \arrow[equal]{r} & M^{\cotensor_C 2} \cotensor_C C \arrow{r} & M^{\cotensor_C 2} \cotensor_C M \arrow{r} & \cdots \\
\S \arrow["\eta"]{u} & & M \cotensor_C \S \arrow["M \cotensor_C \eta"]{u} & & M^{\cotensor_C 2} \cotensor_C \S \arrow["M^{\cotensor_C 2} \cotensor_C \eta"]{u} & & \cdots
\end{tikzcd}
\end{center}
In particular, $C \range 11$ appears as $M \cotensor_C \S$.
\end{proof}

\begin{lemma}\label{BottomAnnularLayerIsAlmostKoszul}
When $\S \to C$ is a pointed coalgebra spectrum so that $\Sch K_* C$ is a formal variety, the spectrum $\S \cotensor_C \S$ has a split filtration of the form \[\Susp^{-1} C\range 11 \to \S \cotensor_C \S \to \S.\]
\end{lemma}
\begin{proof}
Using \Cref{CotensoringPreservesFibers}, we cotensor the resolution sequence $\S \to C \to M$ with the $C$--comodule $\S$ to get the new fiber sequence \[\cdots \to \Susp^{-1} \S \cotensor_C M \to \S \cotensor_C \S \to \S \cotensor_C C \to \S \cotensor_C M \to \cdots.\]  We can use the method of \Cref{CorrectConvergence} to identify $\S \cotensor_C C$ with $\S$, and we can use \Cref{TangentSpectrumInOneCotor} to identify $\S \cotensor_C M$ with $C \range 11$.  This yields the split fiber sequence \[\Susp^{-1} C \range 11 \to \S \cotensor_C \S \to \S. \qedhere\]
\end{proof}

The object $\S \cotensor_C \S = \Omega(\S; C; \S)$ is the subject of Koszul duality.  There is the following general result:
\begin{theorem}[{\cite[Proposition 7.26]{Ching}, \cite[Theorem 2.1.11]{GinzburgKapranov}}]
Let $\mathcal O$ be an operad in $k$--module spectra, where $k$ is some algebra spectrum, and let $A$ be a left--module for the operad $\mathcal O$ (i.e., an $\mathcal O$--algebra).  The arboreal bar construction $B(k; \mathcal O; k)$ provides a co-operad $\mathcal O^\vee$ for which $B(A; \mathcal O; k)$ is a left--comodule, and, dually, the arboreal cobar construction $\Omega(k; \mathcal O^\vee; k)$ provides an operad $(\mathcal O^\vee)^\vee$ for which $\Omega(C; \mathcal O^\vee; k)$ is a left--module.  Additionally, the operad produced in this way from the coassociative co-operad is the associative operad, and the co-operad produced from the associative operad is the coassociative co-operad. \qed
\end{theorem}

\begin{corollary}
The spectrum $\S \cotensor_C \S = (T_\eta C)_+$ is an associative algebra spectrum. \qed
\end{corollary}

There is always a natural map of algebras \[\Tot \Cobar(k; |B(k; A; k)|; k) \to A,\] and when $\sheaf O$ and $A$ are suitably of finite type (e.g., $\sheaf O$ the associative operad and $A$ a finitely generated polynomial algebra), then this map is an isomorphism (i.e., co/bar duality is involutive)~\cite[Proposition 6.4]{Ching}.  It is not clear that such a theorem applies directly here: the category of $\Gamma$--local spectra does not have good finiteness properties, nor do power series algebras.  Instead, we prove the relevant special case of this family of theorems by hand:

\begin{theorem}\label{KoszulDualityForCoalgebraicFormalVarieties}
For $C$ a $\Gamma$--local coalgebraic formal variety spectrum, the natural map \[|B(\S; \S \cotensor_C \S; \S)| \to C\] is a $\Gamma$--local equivalence.
\end{theorem}
\begin{proof}
The two coalgebraic formal schemes \[\Sch K_* C, \quad \Sch K_* |B(\S; \S \cotensor_C \S; \S)|\] are both formal varieties, and the natural map is both a map of coalgebraic schemes and an isomorphism on tangent spaces.  The inverse function theorem for formal varieties thus applies, so that the map is an isomorphism of formal varieties, i.e., a $K_*$--equivalence.
\end{proof}

\begin{remark}
We can use \Cref{KoszulDualityForCoalgebraicFormalVarieties} to think of the attaching map $H$ in
\begin{center}
\begin{tikzcd}
\Susp^{-1} C\range 22 \arrow{r} \arrow[equal]{d} & C \range 11 \arrow{r} \arrow[equal]{d} & C \range 12 \arrow[equal]{d} \\
\Susp^{-1} (T_\eta C)^{\sm 2} \arrow["H"]{r} & T_\eta C \arrow{r} & C\range12
\end{tikzcd}
\end{center}
as a kind of ``Hopf map for $\S \cotensor_C \S$''.
\end{remark}


\section{Projective spaces, formal curves, and Picard elements}
\label{ChapContangentSpaceConstruction}
\label{CalculationForCoalgebraicFVSes}

In this section, we shift attention from theory-building to examples and computations.

\subsection{General features}\label{FormalLines:GeneralFeatures}

We continue to assume that $C$ is a coalgebraic formal curve spectrum, i.e., $K^0 C$ is abstractly isomorphic to a $1$--dimensional power series ring.  In addition to the results of the previous Section, the main consequences of this specialization stem from the following definition and theorem:

\begin{definition}
A spectrum $L$ is said to be \textit{$\Gamma$--locally invertible} if there is some other spectrum $L^{-1}$ with $L \sm L^{-1} \simeq \S$ in the $\Gamma$--local category.  The collection of isomorphism classes of such spectra forms an abelian group, called the \textit{Picard group} of the $\Gamma$--local stable category.
\end{definition}

\begin{theorem}[{\cite[Theorem 1.3]{HMS}}]\label{HMSThm}
A spectrum $L$ is $\Gamma$--locally invertible if and only if $K_* L$ is a $K_*$--line (i.e., a $1$--dimensional $K_*$--vector space, or a $\otimes$--invertible object in the category of $K_*$--vector spaces). \qed
\end{theorem}

\begin{corollary}\label{TetaGivesPicardElements}
When $C$ is a pointed coalgebraic formal curve spectrum, $T_\eta C$ (and, indeed, any annular stratum $C\range{n}{n}$) gives an element of the Picard group of the $\Gamma$--local stable category. \qed
\end{corollary}

This corollary can be viewed in two ways.  First, it can be used to construct elements of the Picard group of the $\Gamma$--local stable category by finding examples of pointed coalgebraic formal curve spectra.  Second, given such a coalgebraic formal curve spectrum, the annular tower can be used to give a kind of cellular decomposition of $C$, by considering the dual tower
\begin{center}
\begin{tikzcd}
* \arrow{r} \arrow{d} & C\range00 \arrow{r} \arrow{d} & C\range01 \arrow{r} \arrow{d} & C\range02 \arrow{r} \arrow{d} & \cdots \\
C\range0\infty \arrow[equal]{r} \arrow{d} & C\range0\infty \arrow[equal]{r} \arrow{d} & C\range0\infty \arrow[equal]{r} \arrow{d} & C\range0\infty \arrow[equal]{r} \arrow{d} & \cdots \\
C\range0\infty \arrow{r} & C\range1\infty \arrow{r} & C\range2\infty \arrow{r} & C\range3\infty \arrow{r} & \cdots
\end{tikzcd}
\end{center}
where each vertical sequence is a fiber sequence.  Using \Cref{TetaGivesPicardElements}, and thinking of Picard elements as ``generalized spheres'', this gives a $\Gamma$--local cellular decomposition of $C$ in terms of $\Gamma$--local Picard elements which are perhaps \emph{not} standard spheres.

\subsection{Classical projective spaces}\label{ExampleTCPinfty}

We now work through the examples where $C$ is taken to be a classical projective space.  In light of \Cref{IdentificationOfPowers}, we will be especially interested in determining the homotopy type of $T_\eta C$.

Let $C = \Susp^\infty_+ \CP^\infty$, with $\eta\co \S \to \Susp^\infty_+ \CP^\infty$ its natural pointing.  In this case, we apply the version of \Cref{CotangentSpectrumDefn} in the global stable category.  However, the global stable category is not local for a field spectrum, so we appeal to the auxiliary spectra $K = H\Q$ and $K = H\F_p$ in order to analyze the coskeletal tower.  Here, we find
\begin{align*}
H\Q_* T_+ \Susp^\infty_+ \CP^\infty & \cong \Susp^2 \Q, \\
(H\F_p)_* T_+ \Susp^\infty_+ \CP^\infty & \cong \Susp^2 \F_p.
\end{align*}
It follows from Sullivan's adelic reconstruction~\cite[Proposition 3.20]{Sullivan} that the integral homology is \[H\Z_* T_+ \Susp^\infty_+ \CP^\infty \cong \Susp^2 \Z,\] and since $T_+ \Susp^\infty_+ \CP^\infty$ is a \emph{connective} spectrum,\footnote{Similar considerations may be used to describe the entire Picard group of the stable category~\cite[Theorem 2.2]{StricklandInterpolation}.} it follows furthermore that its homotopy type is \[T_+ \Susp^\infty_+ \CP^\infty \simeq \S^2.\]  To verify that the annular tower records the cellular decomposition of $\CP^\infty$, we transfer to the dual of the annular tower as in \Cref{FormalLines:GeneralFeatures}.  By connectivity, it again follows that $C\range m n \simeq \CP\range m n$.  The global analysis of $C = \Susp^\infty_+ \HP^\infty$ proceeds similarly.  We may also perform an identical rational analysis of $K(\Q, 2n)$ to produce $T_+ \Susp^\infty_+ K(\Q, 2n) \simeq \S^{2n}_{\Q}$.

We can also address $C = \Susp^\infty_+ \RP^\infty$ and $K = K(\infty) = K_{\G_a} = H\F_2$ (where we use the more general hypothesis available to \Cref{CorrectConvergence}, rather than the even-concentration shortcut).  By similar reasoning to the complex projective case, performing the tangent space construction in the $2$--adic stable category (in fact, because our objects are connective, it suffices to consider the $\G_a$--local category over $\F_2$) yields \[T_+ \Susp^\infty_+ \RP^\infty \simeq (\S^1)^\wedge_2,\] and the annular tower again recovers the cellular decomposition of $\RP^\infty$.

Finally, it's worth remarking that the ambient category chosen to perform the tangent space construction is very important.  It can neither be too localized nor too delocalized:
\begin{itemize}
\item Passing to the $\G_m$--local category (or, generally, the $\Gamma$--local category for $\Gamma$ defined over $k$ of characteristic $2$) factors through $2$--completion.\footnote{\ldots but not through $H\F_2$--localization.}  Ravenel has shown there is an equivalence \[L_{\G_m} \RP\range{8j+1}\infty \simeq L_{\G_m} \S^{-1}\] for any $j$~\cite[Theorem 9.1]{RavenelLocalizationWRTPeriodic}.  Taking $j = 0$, one sees $L_{\G_m} \RP^\infty \simeq L_{\G_m} \S^{-1}$, so its $K(1)$--cohomology is not a power series ring, and it is furthermore too small to have the correct $\Cotor$ groups.  Letting $j$ range, it's also plain that the bar filtration looks wildly different from the behavior of any sort of expected annular filtration.
\item On the other hand, $\RP^\infty$ is $p$--locally acyclic for $p \ne 2$.  It follows that the integral homology of $T_+ \Susp^\infty_+ \RP^\infty$, as computed in the global stable category, will not be a $\Z$--line.
\end{itemize}

\subsection{Determinantal projective space}
\label{SectionAnalyzingTetaForSdet}

We dedicate this subsection to a particularly interesting set of examples, stemming from a calculation of Ravenel and Wilson:
\begin{theorem}[{Ravenel--Wilson~\cite{RavenelWilson}, see also Johnson--Wilson~\cite[Appendix]{JohnsonWilson} and Hopkins--Lurie~\cite[Section 2]{HopkinsLurie}}]\label{RavenelWilsonForKthy}
Take the ambient prime $p$ to satisfy $p \ge 3$.\footnote{The version of this result due to Hopkins--Lurie does not impose this constraint.} There is a Hopf ring isomorphism \[\bigoplus_{q=0}^\infty K_* B^q S^1[p^j] \cong \bigoplus_{q=0}^\infty (K_* BS^1[p^j])^{\wedge q},\] where the target is the free alternating Hopf ring on $K_* BS^1[p^j]$.  Moreover, for fixed $q \ge 1$ the system $\{B^1 S^1[p^j]_K\}_j$ has the structure of a connected $p$--divisible group of dimension $\binom{d-1}{q-1}$, where $B^1 S^1[p^j]_K$ is the $p^j$--torsion subgroup.  The corresponding formal group is given by $(B^q S^1)_K$. \qed
\end{theorem}

Of these examples, the value of $q = 1$ corresponds to the interesting spectrum $\Susp^\infty_+ \CP^\infty$, to which the tools from \Cref{SpectralCotensorSection} apply.  In this other extreme case, $T_+ \Susp^\infty_+ B^d S^1$ has $1$--dimensional $K$--homology, and hence \Cref{TetaGivesPicardElements} and \Cref{IdentificationOfPowers} can again be applied.\footnote{It is interesting to revisit the ideas of \Cref{KoszulDualitySection} in the context of this example: $B^{q-1} S^1$ is the unstable based loopspace of $B^q S^1$, which has very complicated $K$--homology, but $\S \cotensor_{\Susp^\infty_+ B^q S^1} \S$ recovers a ``delooping'' of $B^q S^1$ in the $\Gamma$--local category with simple $K$--homology.  One might also compare this situation with that of $BU(n)$, where the $K$--homology of $\S \cotensor_{\Susp^\infty_+ BU(n)} \S$ agrees with $K_* U(n)$.}  In order to determine the Picard element given by \Cref{TetaGivesPicardElements}, we need a finer invariant than the mere detection property provided by \Cref{HMSThm}.  We find what we seek in the $K(n)$--local version of Morava $E$--homology:

\begin{definition}[{\cite[Definition 8.3]{HoveyStrickland}, \cite[Theorem 12]{StricklandDuality}}]\label{DefnCtsMoravaEThy}
The \textit{continuous Morava $E$--homology} or \textit{$\Gamma$--local Morava $E$--homology} functor is defined by \[E_\Gamma^\vee(X) := \pi_* L_\Gamma (E_\Gamma \sm X).\]  It is valued in topologized modules over $E_{\Gamma*}$ equipped with a continuous coaction of $E_\Gamma^\vee E_\Gamma$.  Equivalently, it can be taken to have values in sheaves on the Lubin--Tate stack for $\Gamma$.  As with the noncontinuous version, we will abbbreviate $E_\Gamma^\vee$ to $E^\vee$ when the fixed formal group $\Gamma$ need not be emphasized.
\end{definition}

\begin{remark}[{\cite[Proposition 8.4.e-f]{HoveyStrickland}}]\label{AcyclicsAndBocksteins}
The notation $E_\Gamma^\vee$ stems from the following result: if $E_\Gamma^*(X)$ is pro-free and even-concentrated, then \[E_\Gamma^\vee(X) = (E_\Gamma^*(X))^\vee\] is its continuous linear dual.  It follows that the acyclics for $E_\Gamma^\vee$--homology agree with the acyclics for $K_\Gamma$--homology, completing the collection mentioned in \Cref{ComputationalToolsSection}~\cite[Proposition 2.5]{HoveyStrickland}.  The main point is that there is a Bockstein spectral sequence taking $K_\Gamma$--homology as input and converging to $E_\Gamma^\vee$--homology.  If the $K_\Gamma$--homology is even--concentrated, this spectral sequence collapses and the $E_\Gamma^\vee$--homology is pro-free on any basis lifting a basis of the $K_\Gamma$--homology.\footnote{This explains the notation: for favorable $X$, $E_\Gamma^\vee(X)$ is the continuous linear dual of $E_\Gamma^*(X)$.}
\end{remark}

\begin{theorem}[{\cite[Proposition 7.5]{HMS}}]\label{EthyDeterminesPic}
There is a factorization
\begin{center}
\begin{tikzcd}
\CatOf{Spectra}_\Gamma \arrow["E_\Gamma^\vee"]{r} \arrow[bend left=20, "K_\Gamma"]{rr} & \CatOf{Modules}_{E_{\Gamma*}, \operatorname{Aut}(\Gamma)} & \CatOf{Modules}_{K_{\Gamma*}} \\
\operatorname{Pic}(\CatOf{Spectra}_\Gamma) \arrow["E_\Gamma^\vee"]{r} \arrow{u} \arrow[bend right=20, "K_\Gamma"]{rr} & \arrow["/\m"]{r} \arrow{u} \CatOf{Lines}_{E_{\Gamma*}, \operatorname{Aut}(\Gamma)} & \CatOf{Lines}_{K_{\Gamma*}} . \arrow{u}
\end{tikzcd}
\end{center}
Both squares are pullback squares: a $\Gamma$--local spectrum is invertible if and only if its continuous $E$--homology is a line, which is true if and only if its $K$--homology is a line.  Additionally, for $p \gg d$ the map labeled $E_\Gamma^\vee$ in the lower-left is an injection on objects.\footnote{Specifically: $2p - 2 \ge d^2$ and $p \ne 2$.} \qed
\end{theorem}

We are moved by this theorem to understand the Morava $E$--homology of Eilenberg--Mac Lane spaces, together with the $\Aut \Gamma$ action.  \Cref{RavenelWilsonForKthy} shows that $K_* B^q S^1[p^j]$ is even-concentrated for any choice of $q$, and hence \Cref{AcyclicsAndBocksteins} shows that $E^\vee B^q S^1[p^j]$ is a pro-free module.  We pursue the Hopf ring structure on these objects, in the hopes that this can be used to extract the $\Aut \Gamma$--action.  Taking a cue from the proof of \Cref{AcyclicsAndBocksteins}, we begin by considering the short exact sequence of coefficients:
\[
0 \to \bigoplus_{\substack{t_0, \ldots, t_{d-1} \ge 0 \\ t_+ = r}} K_*\left\{\left[p^{t_0} \cdots u_{d-1}^{t_{d-1}}\right]\right\} \to E_* / \m^{r+1} \to E_* / \m^r \to 0.
\]
Because we know that $E^\vee B^q S^1[p^j]$ is a pro-free (and thus flat~\cite[Theorem A.9]{HoveyStrickland}) $E_*$--module deforming the original $K_* B^q S^1[p^j]$, we tensor against $E^\vee B^q S^1[p^j]$ to get a new short exact sequence appearing as the top row in \Cref{DiagramOfSexseqs}.  We can also build the free alternating Hopf ring $(E^\vee B S^1[p^j])^{\wedge *}$, and tensoring the above short exact sequence of coefficients with any graded piece of this ring gives the exact sequence (which is not, a priori, left exact) on the bottom row of \Cref{DiagramOfSexseqs}.  Finally, the cup product map $(BS^1[p^j])^{\wedge q} \to B^qS^1[p^j]$ induces a map on homology \[(E^\vee BS^1[p^j])^{\otimes t} \to (E^\vee BS^1[p^j])^{\wedge q} \xrightarrow{\circ} E^\vee B^q S^1[p^j].\]  Bifunctoriality of the tensor product induces the map between these rows by $\circ$--product.

\begin{figure}
\begin{center}
\begin{tikzcd}[row sep=0.5cm, column sep=0.5cm, ampersand replacement=\&]
\& 0 \arrow{d} \\
\displaystyle \bigoplus_{\substack{t_0, \ldots, t_{d-1} \ge 0 \\ t_+ = r}} (K_* B S^1[p^j])^{\wedge q}\left\{\left[p^{t_0} \cdots u_{d-1}^{t_{d-1}}\right]\right\} \arrow{r} \arrow{d} \& \displaystyle \bigoplus_{\substack{t_0, \ldots, t_{d-1} \ge 0 \\ t_+ = r}} (K_* B^q S^1[p^j])\left\{\left[p^{t_0} \cdots u_{d-1}^{t_{d-1}}\right]\right\} \arrow{d} \\
(E^\vee BS^1[p^j])^{\wedge q} \otimes_{E_*} E_* / \m^{r+1} \arrow{r} \arrow{d} \& E^\vee B^q S^1[p^j] \otimes_{E_*} E_* / \m^{r+1} \\
(E^\vee BS^1[p^j])^{\wedge q} \otimes_{E_*} E_* / \m^r \arrow{r} \arrow{d} \& E^\vee B^q S^1[p^j] \otimes_{E_*} E_* / \m^r \arrow{d} \\
0 \& 0.
\end{tikzcd}
\end{center}
\caption{Diagram of short exact sequences in \Cref{RavenelWilsonForEthy}.}
\label{DiagramOfSexseqs}
\end{figure}

\begin{theorem}[{cf.\ \cite[Section 3.4]{HopkinsLurie}}]\label{RavenelWilsonForEthy}
There is an isomorphism of Hopf rings \[\bigoplus_{q=0}^\infty (E^\vee B S^1[p^j])^{\wedge q} \cong \bigoplus_{q=0}^\infty E^\vee B^q S^1[p^j].\]
\end{theorem}
\begin{proof}
We perform an induction on $r$.  In the base case of $r = 1$, \Cref{RavenelWilsonForKthy} gives a chain of isomorphisms
\begin{align*}
(E^\vee B S^1[p^j])^{\wedge q} \otimes_{E_*} E_* / \m^1 & = (K_* BS^1[p^j])^{\wedge q} & \text{(\Cref{AcyclicsAndBocksteins})} \\
& \xrightarrow{\cong} K_* B^q S^1[p^j] & \text{(\Cref{RavenelWilsonForKthy})} \\
& = E^\vee B^q S^1[p^j] / \m^1. & \text{(\Cref{AcyclicsAndBocksteins})}
\end{align*}
This also tells us that the left-hand vertical map in \Cref{DiagramOfSexseqs} is always an isomorphism.  In particular, this map is injective, as is the first nontrivial horizontal map on the second row, so their composite is injective.  It follows that the first horizontal map on the first row \[\bigoplus_{\substack{t_0, \ldots, t_{d-1} \ge 0 \\ t_+ = r}} (K_* BS^1[p^j])^{\wedge q}\left\{\left[p^{t_0} \cdots u_{d-1}^{t_{d-1}}\right]\right\} \to (E^\vee BS^1[p^j] / \m^{r+1})^{\wedge q}\] is also injective, and thus that the top sequence is short exact.

Then, assume that $\circ$--multiplication induces an isomorphism modulo $\m^r$ for some fixed $r$, i.e., that the right-hand vertical map in the above diagram is an isomorphism of modules.  As the left-hand and right-hand vertical maps are isomorphisms, the center map must be as well.  As $t$ varies, the center maps additionally assemble into a map of graded Hopf rings, and so furthermore induce an isomorphism of such.  Induction provides isomorphisms for all $r$, and the Milnor sequence finishes the argument:
\begin{align*}
E^\vee B^q S^1[p^j] & = \mylim{\CatOf{Modules}_{E_*}}{r} (E^\vee B^q S^1[p^j]) / \m^r \\
& = \mylim{\CatOf{Modules}_{E_*}}{r} (E^\vee B S^1[p^j] / \m^r)^{\wedge q} \\
& = (E^\vee B S^1[p^j])^{\wedge q}. \quad\quad \qedhere
\end{align*}
\end{proof}

\begin{remark}
Given this calculation of the Hopf algebra structure on $E^\vee B^q S^1[p^j]$, one might be led to wonder whether this translates into a presentation of $(B^q S^1[p^j])_E$ as the $q${\th} exterior power of the $p^j$--torsion of $\tilde \Gamma$.  This turns out to be true, but making sense of the statement is quite complicated: Buchstaber--Lazarev~\cite{BuchstaberLazarev} and Goerss~\cite{GoerssDieudonne} showed that Dieudonn\'e theory admits a notion of exterior power, Hedayatzadeh~\cite{HedayatzadehGeneralCase,HedayatzadehFieldCase} showed that alternating powers of $1$--dimensional $p$--divisible groups can be constructed intrinsically in special cases (and are compatible with the Dieudonn\'e theory model), and Hopkins and Lurie~\cite[Section 3.5]{HopkinsLurie} give another manual (but coordinate-free!) construction relevant to the case at hand.
\end{remark}

For the remainder of this section, we will need access to only two values of $j$:
\begin{align*}
H_q & = L_\Gamma \Susp^\infty_+ B^d S^1, &
T_q & = L_\Gamma \Susp^\infty_+ B^d S^1[p].
\end{align*}
We are now in a position to identify the $\Aut \Gamma$--representation structure of the $E_*$--line $E^\vee T_+ H_d$, using the $\Aut \Gamma$--equivariant map \[E^\vee T_1^{\times d} \xrightarrow{\circ} E^\vee T_d.\]  It will be helpful to have access to a presentation of the group $\Aut \Gamma$, which arises as the group of units of the algebra $\End \Gamma$.  Writing $S$ for the Frobenius endomorphism of $\Gamma$, one can take the set $\{1, S, \ldots, S^{d-1}\}$ as a $\W(k)$--basis for $\End \Gamma$~\cite[Theorem A2.2.17]{RavenelGreenBook}.

\begin{lemma}
Let $\beta_1 \in E^\vee H_1$ be an element dual to a coordinate on $\tilde\Gamma$.  The algebro-geometric cotangent space of $(H_d)_E$ is spanned by the dual to the element \[\beta_* = S^0 \beta_1 \otimes S^1 \beta_1 \otimes \cdots \otimes S^{d-1} \beta_1 \in E^\vee H_1^{\times d}\] pushed forward along the $\circ$--product map.
\end{lemma}
\begin{proof}
In $K$--theory, this is a rephrasing of results of Ravenel and Wilson~\cite[Theorems 5.7.(d) and 9.2.(a)]{RavenelWilson}.  We then couple these to the lifting methods used elsewhere in this section to conclude the same result for $E$--theory.
\end{proof}

\begin{definition}
Left--multiplication gives a map \[\Aut \Gamma \to GL_{\W(k)}(\End \Gamma)\] and postcomposing with the determinant map defines the \textit{determinant representation}: \[\Aut \Gamma \to GL_{\W(k)}(\End \Gamma) \xrightarrow{\det} \W(k)^\times.\]
\end{definition}

\begin{theorem}[{cf.\ \cite[Proposition 3.21]{Westerland}}]\label{TangentSpectrumOfKZd1IsSdet}
$\Aut \Gamma$ acts on $E^\vee T_+ H_d$ as the determinant representation.
\end{theorem}
\begin{proof}
The cup product map \[E^\vee T_1^{\times d} \xrightarrow{\circ} E^\vee T_d\] is surjective and respects the action of $\Aut \Gamma$, as it is induced by a map of spaces.  There is a K\"unneth formula \[E^\vee T_1^{\times d} \cong (E^\vee T_1)^{\otimes d},\] and the stabilizer group intertwines with K\"unneth maps \[E^\vee A \otimes_{E_*} E^\vee B \xrightarrow{\phi} E^\vee(A \times B)\] as \[g \cdot \phi(a \otimes b) = \phi((g \cdot a) \otimes (g \cdot b)).\]  In turn, an element $g \in \Aut \Gamma$ acts on $\beta_*$ by the formula
\begin{align*}
g \cdot (\circ \beta_*) & = \circ( g \cdot \beta_*) \\
& = \circ(g \cdot (S^0 \beta_1 \otimes \cdots \otimes S^{d-1} \beta_1)) \\
& = \circ(g \cdot (S^0 \otimes \cdots \otimes S^{d-1}) \beta_1) \\
& = \circ(\det g (S^0 \otimes \cdots \otimes S^{d-1}) \beta_1) \\
& = (\det g)(\circ \beta_*). \quad\quad \qedhere
\end{align*}
\end{proof}

\begin{corollary}
For $p \gg d$, $T_+ H_d$ models the determinantal sphere $\S^{\det}$.
\end{corollary}
\begin{proof}
Couple the above with the last part of \Cref{EthyDeterminesPic}.
\end{proof}

\begin{remark}
$\S^{\det}$ is a familiar sight for chromatic homotopy theorists.  Its first appearance was in work of Hopkins and Gross on describing the $\Gamma$--local homotopy type of the Brown--Comenetz dualizing spectrum~\cite[Theorem 6]{HopkinsGrossAnnouncement}.  It has subsequently played a prominent role in the study of chromatic homotopy theory at the height $2$.  For instance, taking $\Gamma$ to be a formal group of height $2$, it has been shown to span the rest of the torsion--free part of the $\Gamma$--local Picard group~\cite[Theorem 8.1]{BehrensRevisited}.  It also has been used to study duality phenomena relating to topological modular forms; see for example work of Behrens~\cite[Proposition 2.4.1]{BehrensModularDescription} and of Stojanoska~\cite[Corollary 13.3]{Stojanoska}.
\end{remark}

Still assuming $p \gg d$, \Cref{EthyDeterminesPic} allows us to determine the $\Gamma$--local homotopy type of $C \range nn$ as well.  We begin with a general result:
\begin{theorem}\label{IdentificationOfPowers}
For $C$ a coalgebraic formal curve spectrum and $p \gg d$, there is a $\Gamma$--local equivalence $C \range nn \simeq (T^*_\eta C)^{\wedge n}$.
\end{theorem}
\begin{proof}
We take as inspiration the sequence of maps of $k$--modules
\begin{center}
\begin{tikzcd}
\frac{\<x\>}{\<x^2\>} \otimes_k \frac{\<x^{n-1}\>}{\<x^n\>} \arrow{r} & \frac{\<x\>}{\<x^2\>} \otimes_{k\ps{x}} \frac{\<x^{n-1}\>}{\<x^n\>} \arrow{r} & \frac{\<x\> \otimes_{k\ps{x}} \<x^{n-1}\>}{\<x^2\> \otimes_{k\ps{x}} \<x^{n-1}\> + \<x\> \otimes_{k\ps{x}} \<x^n\>} \arrow[equal]{r} & \frac{\<x^n\>}{\<x^{n+1}\>}.
\end{tikzcd}
\end{center}
The coalgebraic spectral analogue of the left-most map is the natural truncation
\begin{align*}
C\range 1 1 \cotensor_C C\range{n-1}{n-1} & = \Tot \Omega(C\range 1 1; C; C\range{n-1}{n-1}) \\
& \to \Tot^0 \Omega(C\range 1 1; C; C\range{n-1}{n-1}) \\
& = C\range 1 1 \sm C\range{n-1}{n-1}.
\end{align*}
The sum in the denominator complicates the other map, which can no longer be an isomorphism:
\begin{align*}
C\range n n & = \fib\left( C\range n \infty \to C\range {n+1} \infty \right)
= \fib\left( C \range1 \infty \cotensor_C C\range{n-1} \infty \to C\range 1 \infty \cotensor_C C\range n \infty \right)
= C\range 1 \infty \cotensor_C C\range{n-1}{n-1} \\
& \leftarrow C\range 1 1 \cotensor_C C\range{n-1}{n-1}.
\end{align*}
Nonetheless, upon applying continuous $E$--homology the zig-zag \[E^\vee C\range n n \leftarrow E^\vee (C\range 1 1 \cotensor_C C\range{n-1}{n-1}) \to E^\vee (C\range 1 1 \sm C \range{n-1}{n-1})\] determines the $\Aut \Gamma$--representation structure of $E^\vee C\range n n$ to be that of $E^\vee C\range 1 1 \otimes_{E_*} E^\vee C \range{n-1}{n-1}$.
\end{proof}

\begin{corollary}
For $C = H_d$ and $p \gg d$, there is an equivalence $C\range n n \simeq (\S^{\det})^{\sm n}$. \qed
\end{corollary}

\begin{remark}
It follows from \Cref{TangentSpectrumOfKZd1IsSdet} and \Cref{IdentificationOfPowers} that there is \emph{no} global finite complex with a map to $B^d S^1$ which in $E$--cohomology projects to precisely the $1$--jets.  Namely, $(E_\Gamma)_* B^d S^1$ admits a filtration whose associated-graded is a product of powers of the determinant representation, and any nontrivial map in from a power of the standard representation is obstructed by Schur's lemma.

The reader should compare with the situation with $\CP^\infty$ and ordinary homology, where $\CP^1 \simeq \S^2$ performs this selection of the $1$--jets, but selecting the second annulus is obstructed by $\eta \in \pi_1 \S$, available only after coning off $\CP^1$ or \emph{localizing away from $2$}.  By localizing at $\Gamma$, we have enlarged the variety of spheres available and thus given ourselves more tools by which we can carefully select certain individual homology classes in $\Gamma$--local spectra.

The rest of the annular tower gives a remarkable filtration of the $\Gamma$--local homotopy type of $B^d S^1$, as though it were a cell complex built out of Picard--graded cells with a simple inductive structure.  The global homotopy type $B^d S^1$ of course also comes with a cellular decomposition by global cells---after all, it is presented simplicially by an iterated bar construction---but this information is dramatically more complex.  Morally, passing to the $\Gamma$--local category has simultaneously enlarged our notion of ``cell'' and simplified the homotopy type of a complicated space, at once resulting in a simple pattern not globally visible.
\end{remark}

\subsection{Comparison of $K^{\det}$ with Westerland's $R_d$}\label{ComparisonWithWesterland}

Prompted by the entirety of \Cref{SectionAnalyzingTetaForSdet}, we moved to consider the following umbrella question:

\begin{metaquestion*}
Given a spectrum $X$ with a relationship to $\CP^\infty_+ =: H_1$, a \textit{determinantal analogue} of $X$ is a $\Gamma$--local spectrum with an analogous relationship to \[\XP^\infty_+ := H_d := L_\Gamma \Susp^\infty_+ B^d S^1.\]  Which classically interesting spectra appearing in chromatic homotopy theory have determinantal analogues?  What are they good for?
\end{metaquestion*}

So far, we have found the following determinantal analogues:
\begin{itemize}
    \item The subspectra $\CP^n_+ \subseteq \CP^\infty_+$ have analogues $\XP^n_+ \subseteq \XP^\infty_+$.
    \item The strata $\CP^n_+ / \CP^{n-1}_+ \simeq \S^{2n}$ have analogues $\XP^n_+ / \XP^{n-1}_+ \simeq (\S^{\det})^{\sm n}$.
    \item The $A_\infty$ ring $S^1_+$ has an analogue $(\Susp^{-1} \S^{\det})_+$.
\end{itemize}

\noindent Finding determinantal analogues is no easy feat, and so for this section we limit ourselves to searching for direct consequences of those above.  In the classical case, the indicated composite
\begin{center}
\begin{tikzcd}
\S^0 \arrow{r} \arrow[equal]{d} & \CP^1_+ \arrow{r} \arrow{d} & \CP^2_+ \arrow{r} & \cdots \arrow{r} & \CP^\infty_+ \arrow[leftarrow, thick]{llld}[description]{\beta} \\
\S^0 & \S^2 \arrow[bend left]{u} & \S^4 \arrow[crossing over]{u} & \cdots
\end{tikzcd}
\end{center}
is known as the Bott class, and a theorem of Snaith ties this class to complex $K$-theory:

\begin{theorem}[{\cite{Snaith}}]
There is an equivalence \[\Susp^\infty_+ \CP^\infty [\beta^{-1}] \xrightarrow\simeq KU. \pushQED\qed\qedhere\popQED\]
\end{theorem}

\noindent Using the determinantal analogue of the cellular decomposition as in
\begin{center}
\begin{tikzcd}
\S^0 \arrow{r} \arrow[equal]{d} & \XP^1_+ \arrow{r} \arrow{d} & \XP^2_+ \arrow{r} & \cdots \arrow{r} & \XP^\infty_+, \arrow[leftarrow, thick]{llld}[description]{\beta^{\det}} \\
\S^0 & \S^{\det} \arrow[bend left]{u} & (\S^{\det})^{\sm 2} \arrow[crossing over]{u} & \cdots
\end{tikzcd}
\end{center}
we can construct a determinantal analogue of complex $K$--theory by taking Snaith's theorem as a definition.

\begin{definition}
We set \textit{determinantal $K$--theory} to be \[K^{\det} := \XP^\infty_+ [(\beta^{\det})^{-1}].\]
\end{definition}

\begin{remark}
When $d = 1$, the spectra $H_d$ and $H_1$ agree, as do the homotopy classes $\beta^{\det}$ and $\beta$, hence $K^{\det}$ agrees with $p$--adic $K$--theory $KU^\wedge_p$.
\end{remark}

We're left with the harder part of the meta-question: what use is $K^{\det}$?  Toward this end, Craig Westerland has recently considered a closely related spectrum:
\begin{definition}[{\cite[Definition 3.11]{Westerland}}]
Consider the action of $\Z/p^\times$ on $T_d$ by field multiplication. Averaging this action gives rise to an idempotent in $K$--homology which splits the suspension spectrum as \[T_d \simeq \bigvee_{k=0}^{p-1} Z^{\sm k},\] where $Z$ is a (non-uniquely specified) spectrum with the property $Z^{\sm p} \simeq Z$.  The spectrum $R_d$ is defined by inverting the element of Picard--graded homotopy determined by the composite \[\alpha\co Z \to \bigvee_{k=0}^{p-1} Z^{\otimes k} \simeq T_d \to H_d \xrightarrow{\text{Bockstein}} L_\Gamma \Susp^\infty_+ B^{d+1} \Z_p.\]  Upon picking a coordinate on $\CP^\infty_K$ and applying $K$--homology to this composite, one sees that it selects the dual of an induced coordinate on $(B^{d+1} \Z_p)_K$.
\end{definition}

\begin{theorem}
There is an equivalence $K^{\det} \simeq R_d$.
\end{theorem}
\begin{proof}
We mimic the style of one of Westerland's proofs~\cite[Corollary 3.18]{Westerland}, where he shows that inverting the class $\alpha$ above is equivalent to inverting another class $\rho$, which participates in a $\Gamma$--local diagram
\begin{center}
\begin{tikzcd}
& \Susp^\infty_+ B^{d+1} \Z_p \arrow["\psi^g - g"]{r} \arrow{d} & \Susp^\infty_+ B^{d+1} \Z_p \arrow{d} \\
\S^{\det} \arrow["\rho"]{ru} \arrow["\delta"]{r} & \Susp^\infty_+ B^{d+1} \Z_p[\alpha^{-1}] \arrow["\psi^g - g"]{r} & \Susp^\infty_+ B^{d+1} \Z_p[\alpha^{-1}].
\end{tikzcd}
\end{center}
Here $g$ is a generator of $\Z_p^\times$m $\psi^g$ encodes the natural action of $\Z_p^\times$ on $L_\Gamma \Susp^\infty_+ B^{d+1} \Z_p$ (via the multiplicative action of $\Z_p^\times$ on $\Z_p$), and $\delta$ is defined as the indicated pushforward.  He constructs isomorphisms~\cite[Propositions 3.10 and 3.12]{Westerland}
\begin{center}
\begin{tikzcd}
K_* B^{d+1} \Z_p \arrow["\simeq"]{r} \arrow{d} & C(\Z_p, K_*) \arrow["\operatorname{restrict}"]{d} \\
K_*(\Susp^\infty_+ B^{d+1} \Z_p[\alpha^{-1}]) \arrow["\simeq"]{r} & C(\Z_p^\times, K_*),
\end{tikzcd}
\end{center}
which he uses to calculate that the two maps \[\Susp^\infty_+ B^{d+1} \Z_p[\alpha^{-1}] \xrightarrow{\simeq} \Susp^\infty_+ B^{d+1} \Z_p[\alpha^{-1}][\delta^{-1}] \xleftarrow{\simeq} \Susp^\infty_+ B^{d+1} \Z_p[\rho^{-1}]\] are both equivalences, essentially by calculating the images of $\delta$ and $\rho$ in the ring $C(\Z_p^\times, K_*)$.

Similarly, we will investigate the two maps \[\Susp^\infty_+ B^{d+1} \Z_p[(\beta^{\det})^{-1}] \to \Susp^\infty_+ B^{d+1} \Z_p[\rho^{-1}, (\beta^{\det})^{-1}] \from \Susp^\infty_+ B^{d+1} \Z_p[\rho^{-1}].\]  Using the identification above of $K_* B^{d+1} \Z_p$ as a ring of functions, $\beta^{\det}$ has been handcrafted so that a generator of $K_* \S^{\det}$ is sent by $\beta^{\det}$ according to
\begin{center}
\begin{tikzcd}[row sep=0.2em]
K_* B^{d+1} \Z_p \arrow["\simeq"]{r} & C(\Z_p, K_*) \\
K_* \beta^{\det} \arrow[|->]{r} & (w \mapsto w \pmod{p}).
\end{tikzcd}
\end{center}
Westerland denotes this element of $C(\Z_p, K_*)$ as $f_0$, and he shows that inverting either of $\alpha$ or $\rho$ has the effect of inverting exactly this element~\cite[Proof of Proposition 3.17]{Westerland}.  It follows that both of our two maps are equivalences.
\end{proof}

\begin{remark}[{\cite[Section 3.9]{Westerland}}]
Westerland exposes a variety of remarkable features of $K^{\det}$, the grandest of which is the $E_\infty$--equivalence \[K^{\det} \simeq E_\Gamma^{hS\mathbb G_\Gamma^{\pm}}.\]  Here $\Gamma$ is specifically taken over $k = \F_{p^d}$, the group $\mathbb G_\Gamma$ is the extension of $\Aut \Gamma$ by $\operatorname{Gal}(\F_{p^d} / \F_p)$, and $S\mathbb G_\Gamma^\pm$ is the subgroup of ``special'' elements~\cite[Section 2.2]{Westerland}---i.e., it lies in the fiber sequence \[1 \to S\mathbb G_\Gamma^\pm \to \mathbb G_\Gamma \xrightarrow{\det^\pm} \Z_p^\times \to 1,\] where $\det^\pm$ acts by the determinant on $\Aut \Gamma$ and by $\operatorname{Frob} \mapsto (-1)^{d-1}$ on the Galois component.  It follows that there is a short resolution: \[L_\Gamma \S \to K^{\det} \xrightarrow{\psi - 1} K^{\det},\] where $\psi$ is a certain Adams--type operation inherited from the action of $\Z_p^\times$ on $B^{d+1} \Z_p$ (or, equivalently, the lingering action of $\Z_p^\times$ on the fixed point spectrum)~\cite[Proposition 3.14]{Westerland},~\cite[Theorem 2]{DevinatzHopkinsHFPSS}.  This spectrum has been investigated before, under the names ``Iwasawa extension of the $K(d)$--local sphere'' and ``Mahowald's half-sphere''.
\end{remark}

\begin{remark}[{\cite[Corollary 4.21]{Westerland}}]
Westerland also describes a cellular filtration of $\Susp^\infty_+ B^{d+1} \Z_p$ which bears resemblance to the annular filtration described here, but which is accessed by wholly different means.  He constructs an analogue for $K^{\det}$ of the classical complex $J$--homomorphism, and he then checks that the Thom spectrum of the canonical bundle on $B^{d+1} \Z_p$ has the homotopy type of $(\S^{\det})^{-1} \sm \Susp^\infty B^{d+1} \Z_p$.  This mirrors the following classical fact about projective spaces~\cite[Proposition 4.3]{Atiyah}: \[\operatorname{Thom}(\sheaf L - 1 \downarrow \CP^\infty) \simeq \Susp^{-2} \Susp^\infty \CP^\infty,\] and more generally \[\operatorname{Thom}(m(\sheaf L - 1) \downarrow \CP^{n-m}) \simeq (\CP^1)^{\sm(-m)} \sm \CP\range m n.\]  It would be interesting to know if these filtrations coincide.
\end{remark}

\section{Some open questions}

Before closing, we record some research avenues left open by the present work.

\subsection{Gross--Hopkins dualizing object}

As remarked on in the introduction, the determinantal sphere has previously arisen in connection with the Gross--Hopkins dualizing object $\widehat{\mathbb I}_{\Q/\Z}$ in the $\Gamma$--local category.  Our identification of \(\XP^1\) with $\widehat{\mathbb I}_{\Q/\Z}$ rests on working at a prime $p$ satisfying $p \gg d$, so that we can avail ourselves of \Cref{EthyDeterminesPic}.  At small primes, the picture is murkier: we conjecture that $\XP^1$ disagrees with $\widehat{\mathbb I}_{\Q/\Z}$, just as the Goerss--Henn--Mahowald--Rezk $\S[\det]$ does, but we do not have a proof that this is so---and we furthermore have neither a proof that $\XP^1$ agrees with $\S[\det]$ in the small prime regime nor one that it disagrees.\footnote{The reader may also find the recent paper of Barthel, Beaudry, Goerss, and Stojanoska of interest, where they give yet another construction of an object they call $\S\<\det\>$~\cite{BBGS}.}  Understanding the relationship between these spectra has the opportunity either to shed light on Gross--Hopkins duality or to shed light on exotic Picard groups at greater heights, both of which are interesting prospects.

\subsection{Chromatic splitting}

The chromatic splitting conjecture describes the homotopy groups of such objects as \(\Q \otimes L_{K(d)} \S\), and the Gross--Hopkins dualizing object induces a kind of Poincar\'e duality among these groups.  Westerland has noticed that there is a factorization $\iota'$ of the fundamental class as in
\begin{center}
\begin{tikzcd}
\S^{d+1} \arrow["\iota_{d+1}"]{r} \arrow[densely dotted, "\iota'"']{rd} & L_{K(d)} \Susp^\infty_+ B^{d+1} \Z_p \\
& \XP^1 \arrow["\beta^{\det}"']{u},
\end{tikzcd}
\end{center}
which at large primes can be interpreted as follows:
\begin{align*}
\iota' \in [\S^{d+1}, \XP^1] & = [\S^{d+1}, \Susp^{d-d^2} \widehat{\mathbb I}_{\Q/\Z}] \\
& = [\S^{1+d^2}, \widehat{\mathbb I}_{\Q/\Z}] \\
& = \operatorname{Hom}(\pi_{-1-d^2} M_d \S^0, \Q_p / \Z_p).
\end{align*}

\begin{conjecture}
The class $\iota'$ is nonzero.  It is dual to the top-dimensional class (i.e., the Poincar\'e dualizing class) predicted by the rational chromatic splitting conjecture.
\end{conjecture}

\noindent A positive resolution of this conjecture could give a foothold on a topological explanation for the chromatic splitting conjecture.

The reader should beware that the validity of this line of questioning is quite shaky at at small primes.  It is directly affected by the preceding question: it is unclear how to account for the possibility that $S^{\det} \sm \widehat{\mathbb I}_{\Q/\Z}^{-1}$ is a nontrivial exotic Picard element.  Furthermore, Beaudry has shown the splitting conjecture to be \emph{false} at a small prime (viz., $d = 2$ and $p = 2$), and an adjusted version has not yet been set out.

\subsection{Validity at small primes}

Several of the utility theorems in this paper required the large prime assumption in the proofs presented here, but it is not clear that this assumption is necessary.  In particular, it would be extremely desirable---and not just for the purposes of this paper!---to have a version of \Cref{SadofskysTheoremInCotangentSection} (cf.\ \Cref{SadofskyInverseLimits}) that does not distinguish between the large and small prime cases.

\subsection{Analogues of $BU(n)$}

The construction of $\XP^\infty$ and $\XP^1$ given here, as well as Westerland's construction of determinantal analogues of $L_{K(1)} BU$, of $L_{K(1)} MU$, and of a kind of ``\(J\)--map'', may well be a premonition of a much more extensive story of ``determinantal homotopy theory''.  The most enticing and impactful facet of this would be to find determinantal analogues of the intermediate spaces \(BU(n)\), to understand these as classifying some manner of geometric object, and to be able to interpret them as \emph{spaces} in some kind of unstable determinantal context.  Several more questions long these lines can be found at the end of Westerland's work~\cite[Section 5.2 and Section 7]{Westerland}.

\appendix
\section{Homological algebra of comodules for a Hopf $k$--algebra}
\label{LimitAppendix}

In this section we describe a result, attributed to Hal Sadofsky and expressed in a talk by Mike Hopkins~\cite[Section 14]{MITETheory}, concerning the homology of inverse limits of certain local systems.\footnote{The interested reader can also find related results in a paper of Hovey~\cite{Hovey}.}  None of the material in this section is original; all of it was known to (at least) Hal Sadofsky and Mike Hopkins, and is of ``folk lore'' status among the experts who might be interested.  Sadofsky's theorem is stated as follows:

\begin{theorem}[{\cite{SadofskyHk,SadofskyChromatic}}]\label{SadofskyInverseLimits}
Let $k$ be a field spectrum, i.e., let $k$ be a ring spectrum with $k_*$ a graded field.  (In the case that $k$ is a Morava $K$--theory for $\Gamma$, we require $p \gg \height \Gamma$.)  Furthermore let $X_\alpha$ be a pro-system of $k$--local spectra such that $k_* X_\alpha$ is a Mittag-Leffler system of $k_*$--modules.  There is then a conditionally convergent spectral sequence of signature \[R^* \mylim{\CatOf{Comod}_{k_* k}}{\alpha} k_* X_\alpha \Rightarrow k_* \mylim{\CatOf{Spectra}}{\alpha} X_\alpha,\] where the derived inverse limit on the left is taken in an appropriate category of comodules.
\end{theorem}

\begin{remark}
There is room for improvement in this result: the author sees no reason to require $p \gg d$ except for his own ineptness.  It is quite likely that it holds without this assumption, and a proof of this would strengthen many of the results in this paper.
\end{remark}

\begin{remark}
Locality is \emph{the} essential assumption.  For instance, set $k = H\Q$ and consider the system $\{\S^0 / p^j\}_j$, with maps the natural projections.  The constituent spaces in this system are all rationally acyclic, but the inverse limit is given by the $p$-adic sphere $(\S^0)^\wedge_p$.  Its rational homology is $H\Q_* (\S^0)^\wedge_p \cong \Q_p$, and hence no such convergent spectral sequence can exist.  On the other hand, first rationalizing this system produces the trivial system of zero spectra, and thus the rational homology of the system---which is empty---compares well to the rational homology of the inverse limit---which is also empty.  Noting that $\S^0 / p^j$ is also known as the Moore spectrum $M_0(p^j)$, similar systems can also be constructed for any Morava $K$--theory by employing the generalized Moore spectra of Hopkins and Smith~\cite[Proposition 5.14]{HopkinsSmith}.
\end{remark}

\begin{remark}
In the case $k = H\Q$, the algebra of cooperations is trivial, i.e., $H\Q_* H\Q \cong \Q$.  It follows that the categories of $\Q$--modules and $(H\Q_* H\Q)$--comodules are equivalent, and so Sadofsky's spectral sequence becomes \[R^* \mylim{\CatOf{Modules}_{\Q}}{\alpha} H\Q_* X_\alpha \Rightarrow k_* \mylim{\CatOf{Spectra}}{\alpha} X_\alpha.\]  Because of the Mittag-Leffler condition the higher derived inverse limits vanish, and hence we have the isomorphism \[\mylim{\CatOf{Modules}_{\Q}}{\alpha} H\Q_* X_\alpha \cong k_* \mylim{\CatOf{Spectra}}{\alpha} X_\alpha\] familiar from the Milnor sequence.  In general, we will work to control the first step in this argument (i.e., the derived inverse limits of \emph{comodules}), while retaining hypotheses so that the second step in the argument (i.e., the derived inverse limits of \emph{modules}) remains trivial.
\end{remark}

\begin{remark}\label{ReintroducingMilnorSeqs}
For any sequence of spectra $(X_\alpha)$, the inverse system $\{Y_\alpha := \prod_{\beta \le \alpha} X_\beta\}$ with maps given by projections is Mittag-Leffler.  Then, using the fiber sequence $\lim_{\CatOf{Spectra}}^\alpha X_\alpha \to \prod_\alpha X_\alpha \to \prod_\alpha X_\alpha$, applying $k$--homology gives \[\cdots \to k_* \mylim{\CatOf{Spectra}}{\alpha} X_\alpha \to k_* \left(\prod_\alpha X_\alpha\right) \to k_* \left(\prod_\alpha X_\alpha\right) \to \cdots.\] The middle and right--hand terms of this sequence are calculable by Sadofsky's theorem, even if the inverse system $\{X_\alpha\}$ is itself not Mittag-Leffler, giving some foothold on that case as well.
\end{remark}

\begin{remark}
We remark yet again that computations in the $E_2$--page of this spectral sequence are \emph{exceedingly} complex and promise to shed light on some of the most long-standing conjectures in chromatic homotopy theory.  Our purposes, however, only concern the behavior of this spectral sequence in a maximally degenerate case.
\end{remark}

\subsection{Homological algebra for inverse systems of comodules}

Before engaging in any algebraic topology, we will first make sense of the homological algebra necessary to study derived inverse limits of pro-systems of comodules.  The homological algebra of comodules is well documented elsewhere~\cite[Appendix 1]{RavenelGreenBook}, but the homological algebra of \emph{diagrams} of comodules is more scarce.  Though we are ultimately interested in the case of a Hopf algebroid $(E_*, E_* E)$, in this section we will refer to this pair opaquely as $(A, \Gamma)$ with $\Gamma$ flat over $A$.  To begin, we recall some classical results.

\begin{lemma}[{\cite[Lemma A1.2.1-2]{RavenelGreenBook}}]\label{ClassicalHomologicalAlgebraForComodules}
A comodule $Y$ of the form $Y = \Gamma \otimes_A Y'$ for an $A$--module $Y'$ is said to be an extended comodule.  Such comodules participate in an adjunction\footnote{This adjunction should be thought of as geometrically analogous to the adjunction \[\CatOf{Sets}(X, Y) \cong \text{$G$-}\CatOf{Sets}(G \times X, Y).\]} \[\CatOf{Modules}_A(M, Y') \cong \CatOf{Comod}_{(A, \Gamma)}(M, Y).\]  Furthermore, if $I$ is an injective $A$--module, then $\Gamma \otimes_A I$ is an injective $\Gamma$--comodule, and hence $\CatOf{Comod}_{(A, \Gamma)}$ has enough injectives. \qed
\end{lemma}

\begin{corollary}\label{LimitsAgreeForExtendedThings}
If $\{M_\alpha\}_\alpha$ is a pro-system of $A$--modules and $\{\Gamma \otimes_A M_\alpha\}_\alpha$ is the induced pro-system of extended $\Gamma$--comodules, then there is an isomorphism \[\Gamma \otimes_A \mylim{\CatOf{Modules}_A}{\alpha} M_\alpha \cong \mylim{\CatOf{Comod}_\Gamma}{\alpha} (\Gamma \otimes_A M_\alpha). \pushQED\qed\qedhere\popQED\]
\end{corollary}

Now consider the category $\CatOf{Pro}\CatOf{Comod}_{(A, \Gamma)}$ of pro-systems of comodules.
\begin{lemma}[Sadofsky]\label{SystemsOfComodulesHaveEnoughInjectives}
The category $\CatOf{ProComod}_\Gamma$ has enough injectives.  That is, for $X$ a pro-system of $(A, \Gamma)$--comodules, there is a levelwise injection to an injective system $J$ (which can additionally be taken to be Mittag-Leffler).
\end{lemma}
\begin{proof}
Begin by, for each $\alpha$ in the indexing category, choosing an injection $j'_\alpha: X_\alpha \to J'_\alpha$ with $J'_\alpha$ an injective comodule.  We form a diagram $J$ equipped with a levelwise injection $j: X \to J$ by setting $J_\alpha = \prod_{\beta \le \alpha}^n J'_\beta$, with the structure map $J_\beta^\alpha\co J_\alpha \to J_\beta$ given by projection, and with $j$ defined by \[j_\alpha\co X_\alpha \xrightarrow{\prod_{\beta \le \alpha} j'_\beta \circ X^\alpha_\beta} \prod_{\beta \le \alpha} J'_\beta = J_\alpha.\]

We now check that this diagram has the relevant lifting property:
\begin{center}
\begin{tikzcd}
X \arrow{d} \arrow{r}{j} & J \\ Y \arrow[densely dotted]{ru}[description]{\exists k}
\end{tikzcd}
\end{center}
whenever the vertical arrow is a levelwise injection.  We argue inductively: for an index $\alpha$ without predecessors, the diagram reduces to
\begin{center}
\begin{tikzcd}
X_\alpha \arrow{d} \arrow{r}{j'_\alpha} & J_\alpha \\
Y_\alpha \arrow[densely dotted]{ru}[description]{\exists k'_\alpha},
\end{tikzcd}
\end{center}
which is precisely the diagram describing the classical injectivity condition.  Because $J_\alpha$ was selected to be an injective comodule under $X_\alpha$, such an extension exists.  In the case of a general index $\alpha$, we have the following diagram:
\begin{center}
\begin{tikzcd}[column sep=3.5em]
\prod_{\beta < \alpha} X_\beta \arrow{rrr}{\prod_{\beta' \le \beta < \alpha} j'_{\beta'} \circ X^\beta_{\beta'}} \arrow{ddd} & & & \prod_{\beta < \alpha} J'_\beta \\
& X_\alpha \arrow{d} \arrow{r}{\prod_{\beta \le \alpha} j'_\beta \circ X^\alpha_\beta} \arrow{lu}[description]{\prod_{\beta < \alpha} X^\alpha_\beta} & \prod_{\beta \le \alpha} J'_\beta \arrow{ru}[description]{0_\alpha \times \prod_{\beta < \alpha} 1_\beta} \\
& Y_\alpha \arrow[densely dotted]{ru} \arrow{ld} \\
\prod_{\beta < \alpha} Y_\beta \arrow[bend right]{rrruuu}[description]{\prod_{\beta < \alpha} k_\beta},
\end{tikzcd}
\end{center}
The outer triangle exists and commutes by inductive assumption, and the connecting arrows are part of the data of the pro-systems.  The dashed map is specified by a pair of morphisms $Y_\alpha \to \prod_{\beta < \alpha} J'_\beta$ and $Y_\alpha \to J'_\alpha$.  The former arrow is specified by restriction to the outer triangle.  For the latter arrow, we use the injectivity of $J'_\alpha$ to select an arrow
\begin{center}
\begin{tikzcd}
X_\alpha \arrow{d} \arrow{r}{j'_\alpha} & J'_\alpha \\
Y_\alpha \arrow[densely dotted]{ru}[description]{\exists}.
\end{tikzcd}
\end{center}
\vspace{-1.8\baselineskip}
\end{proof}

\begin{remark}
It is not true that a diagram whose objects consist of levelwise injective comodules is always an injective object in inverse systems of comodules.  The construction above is designed to skirt past questions of compatibility of levelwise lifts.
\end{remark}

Because we have enough injectives, the general machinery of homological algebra applies to produce right--derived functors of the left--exact inverse limit functor \[\mylim{\CatOf{Comod}_\Gamma}{}: \CatOf{ProComod}_{\Gamma} \to \CatOf{Comod}_\Gamma.\]  Moreover, there is a cobar complex which performs this computation.  To see this, we first remark on a consequence of the adjunction of \Cref{ClassicalHomologicalAlgebraForComodules}.

\begin{lemma}[Sadofsky]\label{ProducingFlasqueResns}
If $M_\alpha$ is a Mittag-Leffler pro-system of $A$--modules, then the induced system $M_\alpha \otimes_A \Gamma$ of extended $\Gamma$--comodules is flasque.
\end{lemma}
\begin{proof}
We continue to analyze the adjunction of \Cref{ClassicalHomologicalAlgebraForComodules}, where it will now be convenient to give these functors names: \[U\co \CatOf{Comod}_\Gamma \leftrightarrows \CatOf{Modules}_A \thinspace\thinspace\colon\!C.\]  Because $U$ is exact, there exists a Grothendieck spectral sequence \[E_2^{s,t} \cong (R^s \lim{}_{\CatOf{Comod}_\Gamma}^\alpha) (R^t C) M_\alpha \Rightarrow R^{s+t}(C \circ \lim{}_{\CatOf{Modules}_A}^{\alpha}) M_\alpha.\]  Additionally, because we have assumed $\Gamma$ to be flat as an $A$--module, $C$ is furthermore exact, i.e., $R^{> 0} C = 0$.  The spectral sequence then specializes to an isomorphism \[(R^s \lim{}_{\CatOf{Comod}_\Gamma}^\alpha) (M_\alpha \otimes_A \Gamma) \cong R^s(C \circ \lim{}_{\CatOf{Modules}_A}^{\alpha}) M_\alpha.\]  By restricting attention to Mittag--Leffler pro-systems, we have ensured that $\lim{}_{\CatOf{Modules}_A}^{\alpha}$ is exact, so that $C \circ \lim{}_{\CatOf{Modules}_A}^\alpha$ is the composition of two exact functors and hence has no higher derived functors.  Using the isomorphism, we conclude that the higher derived functors of $\lim{}_{\CatOf{Comod}_\Gamma}^\alpha$ also vanish on this system, i.e., the system is flasque.
\end{proof}

\begin{corollary}[Sadofsky]\label{CobarComputesDerivedLimit}
Write $\Cobar(\Gamma; \overline \Gamma; -)$ for the normalized one--sided cobar cochain complex: \[\Cobar(\Gamma; \overline \Gamma; -)[n] = \Gamma \otimes_A \overline \Gamma^{\otimes_A n} \otimes_A -.\]  If $X_\alpha$ is a Mittag-Leffler pro-system of comodules, then \[R^s \mylim{\CatOf{Comod}_\Gamma}{\alpha} X_\alpha = H^s \mylim{\CatOf{Comod}_\Gamma}{\alpha} \Cobar(\Gamma; \overline \Gamma; X_\alpha).\]
\end{corollary}
\begin{proof}
The cobar complex gives a functorial source of relative injective resolutions of the terms in the pro-system, hence gives a resolution of the pro-system itself.  For a fixed index $n$, the pro-system $\Cobar(\Gamma; \overline \Gamma; M_\alpha)[n]$ is Mittag-Leffler, hence we may apply \Cref{ProducingFlasqueResns} and the theory of flasque resolutions to conclude the indicated formula.
\end{proof}

\subsection{Sadofsky's theorem for finite height Morava $K$--theories}

Now we will discuss a similar theorem communicated to the author by Mike Hopkins~\cite[Section 14]{MITETheory} as a stepping stone toward Sadofsky's theorem for Morava $K$--theories.

\begin{theorem}[Hopkins]\label{HopkinsEnSystem}
Let $E(d)$ be a Johnson--Wilson spectrum and $\{X_\alpha\}_\alpha$ be a system of $E(d)$--local spectra such that $\{E(d)_* X_\alpha\}_\alpha$ is Mittag-Leffler.  There is then a convergent spectral sequence of signature \[R^*\mylim{\CatOf{Comod}_{E(d)_* E(d)}}{\alpha}E(d)_* X_\alpha \Rightarrow E(d)_* \mylim{\CatOf{Spectra}}{\alpha} X_\alpha.\]
\end{theorem}

The proof of this theorem relies on some shorter results, useful in their own right.  Our first subgoal is to show that the $E(d)$--homology of $E(d)$--modules results in an extended comodule, which gives us access to the limit trick in \Cref{LimitsAgreeForExtendedThings}.

\begin{lemma}\label{EHomologyOfModules}
For $M$ an $E(d)$--module spectrum, there is a natural isomorphism \[E(d)_* M \cong E(d)_* E(d) \otimes_{E(d)_*} \pi_* M.\]
\end{lemma}
\begin{proof}
Since $E(d)$ is an $A_\infty$--ring spectrum, there is a strongly convergent spectral sequence describing the tensor of a right $E(d)$--module spectrum $N$ and left $E(d)$--module spectrum $M$: \[\Tor^{E(d)_*}_{*, *}(N, M) \Rightarrow \pi_*(N \sm_{E(d)} M).\]  Taking $N = E(d) \sm E(d)$ and noting that $\pi_* N = E(d)_* E(d)$ is a flat right $E(d)_*$--module, the spectral sequence \[\Tor^{E(d)_*}_{*, *}(E(d)_* E(d), E(d)_* M) \Rightarrow \pi_* ((E(d) \sm E(d)) \sm_{E(d)} M)\] is concentrated on the $0$--line.  Using the freeness of $N$, this collapse gives \[E(d)_* E(d) \otimes_{E(d)_*} \pi_* M \cong \pi_* (E(d) \sm E(d) \sm_{E(d)} M) \cong \pi_* (E(d) \sm M) = E(d)_* M. \qedhere\]
\end{proof}

\begin{corollary}\label{EnIsEasyForModules}
If $\{M_\alpha\}_\alpha$ is a system of $E(d)$--module spectra which is Mittag-Leffler on homotopy, then there is an isomorphism \[E(d)_* \mylim{\CatOf{Modules}_{E(d)}}{\alpha} M_\alpha \cong \mylim{\CatOf{Comod}_{E(d)_* E(d)}}{\alpha} \{E(d)_* M_\alpha\}.\]
\end{corollary}
\begin{proof}
\begin{align*}
E(d)_* \mylim{\CatOf{Modules}_{E(d)}}{\alpha} M_\alpha & \cong E(d)_* E(d) \otimes_{E(d)_*} \pi_* \mylim{\CatOf{Modules}_{E(d)}}{\alpha} M_\alpha & \text{(\Cref{EHomologyOfModules})}\\
& \cong E(d)_* E(d) \otimes_{E(d)_*} \mylim{\CatOf{Modules}_{E(d)_*}}{\alpha} \{\pi_* M_\alpha\}_\alpha & \text{(Mittag-Leffler)} \\
& \cong \mylim{\CatOf{Comod}_{E(d)_* E(d)}}{\alpha} \left(E(d)_* E(d) \otimes_{E(d)_*} \pi_* M_\alpha\right) & \text{(\Cref{LimitsAgreeForExtendedThings})} \\
& \cong \mylim{\CatOf{Comod}_{E(d)_* E(d)}}{\alpha} E(d)_* M_\alpha. & \text{(\Cref{EHomologyOfModules})}
\end{align*}
\end{proof}

Our next goal is to find a topological object to which we can apply \Cref{CobarComputesDerivedLimit}.  This will be a certain Adams--type spectral sequence, and because we have not really used that tool in this paper, we quickly remind the reader of its construction.
\begin{definition}
For a ring spectrum $E$ and spectrum $X$, the following diagram defines an $E$--Adams tower for $X$:
\begin{center}
\begin{tikzcd}
X \arrow{d} & \overline E \sm X \arrow{l} \arrow{d} & \overline E^{\sm 2} \sm X \arrow{l} \arrow{d} & \overline E^{\sm 3} \sm X \arrow{l} \arrow{d} & \cdots \arrow{l} \\
E \sm X & E \sm \overline E \sm X & E \sm \overline E^{\sm 2} \sm X & E \sm \overline E^{\sm 3} \sm X & \cdots,
\end{tikzcd}
\end{center}
where $\overline E$ participates in the fiber sequence \[\overline E \to S \xrightarrow{\eta_E} E.\]  In good cases, this spectral sequence converges conditionally to homotopy of the $E$--nilpotent completion of $X$~\cite[Proposition 6.3]{Bousfield}.  In the case of a smashing localization, the homotopy of the $E$--nilpotent completion of $X$ agrees with that of the Bousfield $E$--localization of $X$~\cite[Corollary 6.13]{Bousfield}.  In better cases still, the $E^2$--page of the spectral sequence can be identified~\cite[Theorem 2.2.11]{RavenelGreenBook} as \[E^2_{*, *} \cong \Cotor^{E_* E}_{*, *}(E_*, E_* X).\]
\end{definition}

We are now in a position to construct the spectral sequence in Hopkins's inverse limit theorem.

\begin{definition}\label{DefnHopkinsInverseSSeq}
Given an pro-system of spectra $X_\alpha$, we may construct a family of interlocking fiber sequences
\begin{center}
\begin{tikzcd}[column sep=0.1cm]
\mylim{\CatOf{Spectra}}{\alpha} X_\alpha \arrow{d} & \mylim{\CatOf{Spectra}}{\alpha} \overline{E(d)} \sm X_\alpha \arrow{l} \arrow{d} & \mylim{\CatOf{Spectra}}{\alpha} \overline{E(d)}^{\sm 2} \sm X_\alpha \arrow{l} \arrow{d} & \cdots \arrow{l} \\
\mylim{\CatOf{Spectra}}{\alpha} E(d) \sm X_\alpha & \mylim{\CatOf{Spectra}}{\alpha} E(d) \sm \overline{E(d)} \sm X_\alpha & \mylim{\CatOf{Spectra}}{\alpha} E(d) \sm \overline{E(d)}^{\sm 2} \sm X_\alpha & \cdots,
\end{tikzcd}
\end{center}
which upon applying $E(d)$--homology gives a spectral sequence with $E^1$--page \[E^1_{*, t} = E(d)_* \mylim{\CatOf{Spectra}}{\alpha} \left( E(d) \sm \overline{E(d)}^{\sm t} \sm X_\alpha \right)\] and with target $E(d)_* \mylim{\CatOf{Spectra}}{\alpha} X_\alpha$.  If we additionally assume that the pro-system consists of $E(d)$--local spectra such that the induced system $\{E(d)_* X_\alpha\}_\alpha$ is Mittag-Leffler, we may combine the preceding corollaries to compute \[E(d)_* \mylim{\CatOf{Spectra}}{\alpha} \left( E(d) \sm \overline{E(d)}^{\sm t} \sm X_\alpha \right) \cong\] \[E(d)_* E(d) \otimes_{E(d)_*} \left(E(d)_* \overline{E(d)}\right)^{\otimes_{E(d)_*} t} \otimes_{E(d)_*} \mylim{\CatOf{Modules}_{E(d)_*}}{\alpha} E(d)_* X_\alpha\] and hence by \Cref{CobarComputesDerivedLimit} \[E^2_{*, *} \cong R^* \mylim{\CatOf{Comod}_{E(d)_* E(d)}}{\alpha} E(d)_* X_\alpha.\]
\end{definition}

\begin{lemma}\label{SphereIsPrenilpotent}
The $E(d)$--Adams resolution for the $E(d)$--local sphere $L_{E(d)} \S^0$ is equivalent to a finite--dimensional cosimplicial resolution by $E(d)$--module spectra.  (In particular, the $E(d)$--Adams spectral sequence has a horizontal vanishing line.)
\end{lemma}
\begin{proof}
We recall the following terminology: a spectrum is \textit{$E$--nilpotent} if it belongs to the thick $\otimes$--ideal generated by $E$ and \textit{$E$--prenilpotent} if it is $E$--locally equivalent to an $E$--nilpotent spectrum~\cite[Definition 7.1.6]{RavenelOrangeBook}.  Work of Ravenel~\cite[Lemmas 8.3.7 and 8.3.1]{RavenelOrangeBook} gives an $L_{E(d)} BP$--prenilpotent finite spectrum $F$ whose ordinary homology is torsion--free.  Since finite prenilpotent spectra form a thick subcategory, it follows from their classification~\cite[Theorem 9]{HopkinsSmith} that $\S^0$ is $L_{E(d)} BP$--prenilpotent.  Since $L_{E(d)} BP$ and $E(d)$ share a Bousfield class~\cite[Theorem 7.3.2b and Lemma 8.1.4]{RavenelOrangeBook}, it follows that $L_{E(d)} \S^0$ is thus $E(d)$--nilpotent and hence has a finite $E(d)$--Adams resolution.
\end{proof}

\begin{corollary}\label{FiniteEnResolutions}
Every $E(d)$--local spectrum $X$ has a functorial finite--dimensional cosimplicial resolution by $E(d)$--module spectra.  The length of the resolution is independent of $X$ and dependent only on the prime $p$ and height $d$.
\end{corollary}
\begin{proof}
Since $L_{E(d)}$ is a smashing localization~\cite[Theorem 7.5.6]{RavenelOrangeBook}, we can smash the finite resolution for $L_{E(d)} \S^0$ guaranteed by \Cref{SphereIsPrenilpotent} with $X$.
\end{proof}

\begin{proof}[{Proof of \Cref{HopkinsEnSystem}}]
Having constructed the relevant spectral sequence in \Cref{DefnHopkinsInverseSSeq}, we need only address convergence.  \Cref{FiniteEnResolutions} shows that the homotopy inverse system in \Cref{DefnHopkinsInverseSSeq} is weakly equivalent to a finite inverse system.  It follows that, upon applying $E(d)$--homology, the resulting spectral sequence is concentrated in a finite horizontal band and hence is strongly convergent to $E(d)_* \mylim{\CatOf{Spectra}}{\alpha} X_\alpha$.
\end{proof}

\begin{lemma}[{\cite[Footnote 1]{Devinatz}, cf.\ \cite[Proposition 6.5]{HoveyStrickland}, \cite[Proposition 7.5]{HMS}}]\label{ExistenceOfSTComplexes}
For $p \gg d$, there exists an $E(d)$--local Smith--Toda complex $V(d-1)$, which is a finite complex with the property $E(d) \sm V(d-1) \simeq K(d)$.
\end{lemma}
\begin{proof}
We actually demonstrate the existence of $E(d)$--local Smith--Toda complexes $V(j)$ for $j < d$, proceeding by induction from $V(-1) = \S$.  The homology groups $E(d)_* V(j)$ are concentrated in degrees which are multiples of $|v_1| = 2(p-1)$, and so we may additionally take $p > 2$ and apply Ravenel's sparseness result~\cite[Proposition 4.4.2]{RavenelGreenBook} to conclude that the $E(d)$--based Adams spectral sequence \[E_2^{s,t} = \Ext^{s,t}_{E(d)_* E(d)}(E(d)_* V(j), E(d)_* V(j))\] is concentrated in degrees satisfying $2(p-1) \mid t$, hence that the only nontrivial differentials lie on page indices $r$ satisfying $2(p-1) \mid r-1$.  At the same time, the finite resolution length stated qualitatively in \Cref{FiniteEnResolutions} can be shown quantitatively to be $d^2$~\cite[Theorem 6.2.10.(b)]{RavenelGreenBook}, so that if $2(p-1) > d^2$ there are then no $r$ satisfying $2(p-1) \mid r - 1$ with both the source and target of $d_r$ lying within the $t$--band $0 \le t \le d^2$.  It follows that the $E(d)$--based Adams spectral sequence collapses, hence that the element \[v_{j+1} \in \operatorname{Ext}^{|v_{j+1}|, 0}_{E(d)_* E(d)}(E(d)_* V(j), E(d)_* V(j))\] survives, and hence that the complex $V(j + 1) = V(j) / v_{j+1}$ exists.
\end{proof}

From here, we can conclude Sadofsky's theorem for $K(d)$.

\begin{proof}[{Proof of Sadofsky's theorem for $k = K(d)$}]
Use \Cref{ExistenceOfSTComplexes} to extract an $E(d)$--local Smith--Toda complex $V(d-1)$.  Replacing the system $\{X_\alpha\}_\alpha$ by $\{X_\alpha \sm V(d-1)\}_\alpha$ and applying Sadofsky's result for $E(d)$--homology yields the desired spectral sequence.\footnote{For a while, the author thought that because $K(d)$ admits a finite resolution in $E(d)$--module spectra, this proof would go through at all primes, but it seems to be a dead end.}
\end{proof}

\begin{remark}
An odd wrinkle of this construction is that the inverse limit of $\{K(d)_* X_\alpha\}_\alpha$ is still taken in the category of $E(d)_* E(d)$--comodules, \emph{not} of $K(d)_* K(d)$--comodules.  However, the $E(d)_* E(d)$--comodule structure of $K(d)_* X_\alpha$ factors through the Hopf algebroid $(K(d)_*, \Gamma')$, where $\Gamma'$ is given by
\begin{align*}
\Gamma' = K(d)_* \otimes_{BP_*} BP_* BP \otimes_{BP_*} K(d)_* & = K(d)_*[t_1, t_2, \ldots] / (v_d t_j^{p^d} - v_d^{p^j} t_j \mid j > 0) \\
& \subsetneq K(d)_* K(d) = \Gamma' \otimes \Lambda[\tau_1, \ldots, \tau_{d-1}].
\end{align*}
The reader should compare these stray $\tau_*$ cooperations with the Bockstein operations appeaing in the proof of \cite[Proposition 8.4.e-f]{HoveyStrickland}.
\end{remark}

\subsection{Sadofsky's theorem for ordinary homology with field coefficients}

In the case $k = HK$ for a field $K$ of positive characteristic, all of the above constructions can be re-done to produce a derived inverse limit spectral sequence for $HK$--homology.  However, our convergence argument fails badly, as the $p$--complete sphere is no longer finitely resolvable by $HK$--module spectra---after all, the $HK$--Adams spectral sequence has both an infinite tower and a vanishing line of slope $1$, rather than the horizontal vanishing line present in the $E(d)$--Adams spectral sequence.  It follows that the resulting spectral sequence is merely conditionally convergent, and additional hypotheses on the system $\{X_\alpha\}_\alpha$ are required to do any better.  In spite of the lack of topological finiteness, Sadofsky has proven the following theorem whose proof we will not recount:

\begin{theorem}[{\cite{SadofskyHk}}]
The $HK$--based inverse limit spectral sequence converges strongly in the case that $R^s \mylim{\CatOf{Modules}_{\mathcal A_*}}{\alpha} H_*(X_\alpha) = 0$ for $s \gg 0$. \qed
\end{theorem}

%
%
%
\bibliographystyle{gtart}
\bibliography{annular}

\begin{thebibliography}{}
\providecommand\bibmarginpar{\leavevmode\marginpar}
\def\urlstyle#1{{\tt #1}}

\bibitem{Atiyah}
\textbf{M\,F Atiyah}, \emph{Thom complexes}, Proc. London Math. Soc. (3) 11
  (1961) 291--310

\bibitem{BBGS}
\textbf{T {Barthel}}, \textbf{A {Beaudry}}, \textbf{P\,G {Goerss}}, \textbf{V
  {Stojanoska}}, \emph{{C}onstructing the determinant sphere using a {T}ate
  twist}, arXiv e-prints  (October 2018) arXiv:1810.06651

\bibitem{BarthelBeaudryPeterson}
\textbf{T Barthel}, \textbf{A Beaudry}, \textbf{E Peterson}, \textbf{H
  Sadofsky}, \emph{The homology of inverse limits and the algebraic chromatic
  splitting conjecture}, Unpublished  (2017)

\bibitem{BehrensModularDescription}
\textbf{M Behrens}, \href{http://dx.doi.org/10.1016/j.top.2005.08.005} {\emph{A
  modular description of the {$K(2)$}-local sphere at the prime 3}}, Topology
  45 (2006) 343--402

\bibitem{BehrensRevisited}
\textbf{M Behrens}, \href{http://dx.doi.org/10.1016/j.aim.2012.02.023}
  {\emph{The homotopy groups of {$S_{E(2)}$} at {$p\geq 5$} revisited}}, Adv.
  Math. 230 (2012) 458--492

\bibitem{Boardman}
\textbf{J\,M Boardman}, \emph{Conditionally convergent spectral sequences},
  from ``Homotopy invariant algebraic structures ({B}altimore, {MD}, 1998)'',
  Contemp. Math. 239, Amer. Math. Soc., Providence, RI (1999)  49--84

\bibitem{Bousfield}
\textbf{A\,K Bousfield}, \href{http://dx.doi.org/10.1016/0040-9383(79)90018-1}
  {\emph{The localization of spectra with respect to homology}}, Topology 18
  (1979) 257--281

\bibitem{BuchstaberLazarev}
\textbf{V Buchstaber}, \textbf{A Lazarev},
  \href{http://dx.doi.org/10.2140/agt.2007.7.529} {\emph{Dieudonn\'e modules
  and {$p$}-divisible groups associated with {M}orava {$K$}-theory of
  {E}ilenberg-{M}ac {L}ane spaces}}, Algebr. Geom. Topol. 7 (2007) 529--564

\bibitem{Ching}
\textbf{M Ching}, \href{http://dx.doi.org/10.2140/gt.2005.9.833} {\emph{Bar
  constructions for topological operads and the {G}oodwillie derivatives of the
  identity}}, Geom. Topol. 9 (2005) 833--933 (electronic)

\bibitem{Demazure}
\textbf{M Demazure}, \emph{Lectures on {$p$}-divisible groups}, volume 302 of
  \emph{Lecture Notes in Mathematics}, Springer-Verlag, Berlin (1986)Reprint of
  the 1972 original

\bibitem{Devinatz}
\textbf{E\,S Devinatz}, \href{http://dx.doi.org/10.1016/j.aim.2008.06.020}
  {\emph{Towards the finiteness of {$\pi_\ast L_{K(n)}S^0$}}}, Adv. Math. 219
  (2008) 1656--1688

\bibitem{DevinatzHopkinsHFPSS}
\textbf{E\,S Devinatz}, \textbf{M\,J Hopkins},
  \href{http://dx.doi.org/10.1016/S0040-9383(03)00029-6} {\emph{Homotopy fixed
  point spectra for closed subgroups of the {M}orava stabilizer groups}},
  Topology 43 (2004) 1--47

\bibitem{GinzburgKapranov}
\textbf{V Ginzburg}, \textbf{M Kapranov},
  \href{http://dx.doi.org/10.1215/S0012-7094-94-07608-4} {\emph{Koszul duality
  for operads}}, Duke Math. J. 76 (1994) 203--272

\bibitem{GoerssDieudonne}
\textbf{P\,G Goerss}, \href{http://dx.doi.org/10.1090/conm/239/03600}
  {\emph{Hopf rings, {D}ieudonn\'e modules, and {$E_*\Omega^2S^3$}}}, from
  ``Homotopy invariant algebraic structures ({B}altimore, {MD}, 1998)'',
  Contemp. Math. 239, Amer. Math. Soc., Providence, RI (1999)  115--174

\bibitem{GHMR}
\textbf{P Goerss}, \textbf{H-W Henn}, \textbf{M Mahowald}, \textbf{C Rezk},
  \href{http://dx.doi.org/10.1112/jtopol/jtu024} {\emph{On {H}opkins' {P}icard
  groups for the prime 3 and chromatic level 2}}, J. Topol. 8 (2015) 267--294

\bibitem{Haugseng}
\textbf{R Haugseng}, \href{http://dx.doi.org/10.2140/gt.2017.21.1631}
  {\emph{The higher {M}orita category of {$\Bbb{E}_n$}--algebras}}, Geom.
  Topol. 21 (2017) 1631--1730

\bibitem{HedayatzadehFieldCase}
\textbf{S\,M\,H Hedayatzadeh},
  \href{http://dx.doi.org/10.1016/j.jnt.2013.10.023} {\emph{Exterior powers of
  {$\pi$}-divisible modules over fields}}, J. Number Theory 138 (2014) 119--174

\bibitem{HedayatzadehGeneralCase}
\textbf{S\,M\,H Hedayatzadeh},
  \href{http://jtnb.cedram.org/item?id=JTNB_2015__27_1_77_0} {\emph{Exterior
  powers of {L}ubin-{T}ate groups}}, J. Th\'eor. Nombres Bordeaux 27 (2015)
  77--148

\bibitem{HopkinsGrossAnnouncement}
\textbf{M\,J Hopkins}, \textbf{B\,H Gross},
  \href{http://dx.doi.org/10.1090/S0273-0979-1994-00438-0} {\emph{The rigid
  analytic period mapping, {L}ubin-{T}ate space, and stable homotopy theory}},
  Bull. Amer. Math. Soc. (N.S.) 30 (1994) 76--86

\bibitem{HopkinsLurie}
\textbf{M\,J Hopkins}, \textbf{J Lurie}, \emph{Ambidexterity in {$K(n)$}-local
  stable homotopy theory}
  \texttt{http://www.math.harvard.edu/{\textasciitilde}lurie/papers/Ambidexterity.pdf},
  Accessed: 2015-01-09

\bibitem{HMS}
\textbf{M\,J Hopkins}, \textbf{M Mahowald}, \textbf{H Sadofsky},
  \emph{Constructions of elements in {P}icard groups}, from ``Topology and
  representation theory ({E}vanston, {IL}, 1992)'', Contemp. Math. 158, Amer.
  Math. Soc., Providence, RI (1994)  89--126

\bibitem{HopkinsSmith}
\textbf{M\,J Hopkins}, \textbf{J\,H Smith},
  \href{http://dx.doi.org/10.2307/120991} {\emph{Nilpotence and stable homotopy
  theory. {II}}}, Ann. of Math. (2) 148 (1998) 1--49

\bibitem{Hovey}
\textbf{M Hovey}, \href{http://dx.doi.org/10.1017/S0017089507003369} {\emph{The
  generalized homology of products}}, Glasg. Math. J. 49 (2007) 1--10

\bibitem{HoveyStrickland}
\textbf{M Hovey}, \textbf{N\,P Strickland}, \emph{Morava {$K$}-theories and
  localisation}, Mem. Amer. Math. Soc. 139 (1999) viii+100

\bibitem{JohnsonWilson}
\textbf{D\,C Johnson}, \textbf{W\,S Wilson},
  \href{http://dx.doi.org/10.2307/2374422} {\emph{The {B}rown-{P}eterson
  homology of elementary {$p$}-groups}}, Amer. J. Math. 107 (1985) 427--453

\bibitem{Lurie}
\textbf{J Lurie}, \emph{{H}igher {A}lgebra. 2016} Preprint, available at
  http://www.math.harvard.edu/{\textasciitilde}lurie

\bibitem{MITETheory}
\textbf{E Peterson}, \emph{Report on {$E$}--Theory Conjectures
  Seminar}{h}ttp://chromotopy.org/latex/misc/mit-ethy.pdf, Accessed: 2016-06-02

\bibitem{RavenelLocalizationWRTPeriodic}
\textbf{D\,C Ravenel}, \href{http://dx.doi.org/10.2307/2374308}
  {\emph{Localization with respect to certain periodic homology theories}},
  Amer. J. Math. 106 (1984) 351--414

\bibitem{RavenelGreenBook}
\textbf{D\,C Ravenel}, \emph{Complex cobordism and stable homotopy groups of
  spheres}, volume 121 of \emph{Pure and Applied Mathematics}, Academic Press
  Inc., Orlando, FL (1986)

\bibitem{RavenelOrangeBook}
\textbf{D\,C Ravenel}, \emph{Nilpotence and periodicity in stable homotopy
  theory}, volume 128 of \emph{Annals of Mathematics Studies}, Princeton
  University Press, Princeton, NJ (1992)Appendix C by Jeff Smith

\bibitem{RavenelWilson}
\textbf{D\,C Ravenel}, \textbf{W\,S Wilson},
  \href{http://dx.doi.org/10.2307/2374093} {\emph{The {M}orava {$K$}-theories
  of {E}ilenberg-{M}ac {L}ane spaces and the {C}onner-{F}loyd conjecture}},
  Amer. J. Math. 102 (1980) 691--748

\bibitem{SadofskyHk}
\textbf{H Sadofsky}, \emph{The Homology of Inverse Limits of
  Spectra}Unpublished

\bibitem{SadofskyChromatic}
\textbf{H Sadofsky}, \emph{Towards the chromatic splitting
  conjecture}Unpublished

\bibitem{Snaith}
\textbf{V Snaith}, \href{http://dx.doi.org/10.1017/S0305004100058205}
  {\emph{Localized stable homotopy of some classifying spaces}}, Math. Proc.
  Cambridge Philos. Soc. 89 (1981) 325--330

\bibitem{Stojanoska}
\textbf{V Stojanoska}, \emph{Duality for topological modular forms}, Doc. Math.
  17 (2012) 271--311

\bibitem{StricklandConjecture}
\textbf{N\,P Strickland}, \emph{A conjecture about the {$K(n)$}--local {P}icard
  group}Unpublished

\bibitem{StricklandInterpolation}
\textbf{N\,P Strickland}, \href{http://dx.doi.org/10.1017/CBO9780511526312.010}
  {\emph{On the {$p$}-adic interpolation of stable homotopy groups}}, from
  ``Adams {M}emorial {S}ymposium on {A}lgebraic {T}opology, 2 ({M}anchester,
  1990)'', London Math. Soc. Lecture Note Ser. 176, Cambridge Univ. Press,
  Cambridge (1992)  45--54

\bibitem{StricklandFSFG}
\textbf{N\,P Strickland}, \emph{Formal schemes and formal groups}, from
  ``Homotopy invariant algebraic structures ({B}altimore, {MD}, 1998)'',
  Contemp. Math. 239, Amer. Math. Soc., Providence, RI (1999)  263--352

\bibitem{StricklandDuality}
\textbf{N\,P Strickland},
  \href{http://dx.doi.org/10.1016/S0040-9383(99)00049-X} {\emph{Gross-{H}opkins
  duality}}, Topology 39 (2000) 1021--1033

\bibitem{Sullivan}
\textbf{D\,P Sullivan}, \emph{Geometric topology: localization, periodicity and
  {G}alois symmetry}, volume~8 of \emph{$K$-Monographs in Mathematics},
  Springer, Dordrecht (2005)The 1970 MIT notes, Edited and with a preface by
  Andrew Ranicki

\bibitem{Westerland}
\textbf{C Westerland}, \href{http://dx.doi.org/10.2140/gt.2017.21.1033}
  {\emph{A higher chromatic analogue of the image of {$J$}}}, Geom. Topol. 21
  (2017) 1033--1093

\end{thebibliography}

\end{document}